\documentclass{article}

\usepackage[left=1in,right=1in,top=1in,bottom=1in]{geometry}

\usepackage{enumerate}
\usepackage[pdftex]{graphicx}
\usepackage{url}
\usepackage{wrapfig,amssymb,amsfonts,amsmath,amsthm,ulem,color}
\usepackage{caption}
\usepackage{bbm}
\usepackage{subcaption}
\usepackage[table]{xcolor}
\usepackage{epsfig}
\definecolor{lightgray}{gray}{0.9}

\usepackage{cite}

\newtheorem{definition}{Definition}

\newtheorem{remark}{Remark}
\newtheorem{lemma}{Lemma}
\newtheorem{theorem}{Theorem}

\newtheorem{conjecture}{Conjecture}

\def\R{\mathbb R}

\def\X{\mathbb X}

\def\A{\mathcal{A}}

\def\C{\mathcal C}
\def\S{\mathcal S}
\def\Re{\mathcal R}

\def\X2{\mathbf{X^{(2)}}}
\def\X{\mathbf{X}}

\def\I{\mathcal{I}}

\title{Stochastic analysis of biochemical reaction networks with absolute concentration robustness}

\author{David F. Anderson\thanks{Department of Mathematics, University of
  Wisconsin, Madison;  anderson@math.wisc.edu.},
 \and
Germ\'an A. Enciso\thanks{Department of Mathematics, University of California, Irvine; enciso@uci.edu.}, 
\and Matthew D. Johnston\thanks{Department of Mathematics, University of
  Wisconsin, Madison; mjohnston3@wisc.edu.}}

\begin{document}

\maketitle

\begin{abstract}
  It has recently been shown that structural conditions on the reaction network, rather than a `fine-tuning' of system parameters, often suffice to impart `absolute concentration robustness' on a wide class of biologically relevant, deterministically modeled mass-action systems [Shinar and Feinberg, Science, 2010].  We show here that fundamentally different conclusions about the long-term behavior of such systems are reached if the systems are instead modeled with stochastic dynamics and a discrete state space. Specifically, we characterize a large class of models that exhibit convergence to a positive robust equilibrium in the deterministic setting, whereas trajectories of the corresponding stochastic models are necessarily absorbed by a set of states that reside on the boundary of the state space, i.e. the system undergoes an extinction event.  If the time to extinction is large relative to the relevant time-scales of the system, the process will appear to settle down to a stationary distribution long before the inevitable extinction will occur.  This quasi-stationary distribution is considered  for two systems taken from the literature, and results consistent with absolute concentration robustness are recovered by showing that the quasi-stationary distribution of the robust species approaches a Poisson distribution.       
  \end{abstract}

\begin{flushleft}
\end{flushleft}

\section{Introduction}

  The interaction networks of chemical reaction systems of cellular processes are notoriously complex.  Despite this, hidden within the complexity there are often underlying structures that, if properly quantified, give great insight into the dynamical or stationary behavior of the system.  In this vein, Shinar and Feinberg have presented conditions on the structure of biochemical reaction networks that are sufficient to guarantee {\it absolute concentration robustness} (ACR) on a particular species of the network \cite{Sh-F1}.  When the dynamics of the system are modeled using ordinary differential equations with mass-action kinetics, a species is said to possess ACR if its concentration has the same value at every positive equilibrium concentration permitted by the system of equations, regardless of total molar concentrations.
Such a property, which allows  cells to respond in a uniform, predictable way given varying environments, is fundamental to many biological processes, including signal transduction cascades and gene regulatory networks \cite{B-F1}.


 In the two-component EnvZ/OmpR osmoregulatory signaling system, for example, it is important that the amount of phosphorylated OmpR, OmpR-P, which regulates the transcription of the porins OmpF and OmpC, is kept within tight bounds.
It has been observed that while OmpR-P is sensitive to the availability of ADP and ATP, it is relatively insensitive to changes in the overall concentrations of the signaling proteins EnvZ and OmpR \cite{B-G1}. Concentration robustness has also been experimentally observed and studied in the two-component KdpD/KdpE \cite{K-H-C-J-G1}, PhoQ/PhoP \cite{M-G1}, and CpxA/CpxR \cite{S-G1} signaling systems, and in the IDHKP-IDH glyoxylate bypass regulation system \cite{LaPorte1,Shinar2009-21}. Structural sources of robustness have also been identified in the bacterial chemotaxis pathway \cite{Alon19991,Barkai19971,S-W-S-K1}.

Consistent with these empirical results, it has been shown that deterministic mathematical models (i.e. ODE models) of the EnvZ/OmpR system yield equilibrium concentrations of OmpR-P that are stable and
do not depend on the overall concentrations of either EnvZ or OmpR \cite{Shinar20071}. This phenomenon of equilibrium concentrations being independent of total molar concentrations was the basis of \cite{Sh-F1} where Shinar and Feinberg presented their conditions on the structure of biochemical reaction networks which are sufficient to guarantee ACR on a particular species of the network.
Shinar and Feinberg relate the capacity of a network to exhibit ACR to a structural parameter called the {\it deficiency}, which is well-studied in {\it Chemical Reaction Network Theory} (CRNT) \cite{F11,H1,H-J11}. They do not consider stability of such equilibria directly, but they do apply their results to several mass-action models of biochemical networks for which stability is known, including the EnvZ/OmpR signal transduction network and the IDHKP-IDH glyoxylate bypass regulatory system.
An interesting recent addition to this framework is the  work by Karp et al. \cite{Karp1}, where the authors
carry out a linear analysis of formal expressions in a reaction network to find ACR and more general steady state invariants.

In the present work we consider stochastically modeled systems satisfying essentially the same network conditions
used by Shinar and Feinberg, and we show that strikingly different conclusions are reached pertaining to the long-term dynamics of the systems. For a wide class of biochemical reaction networks with stable ACR equilibria, we show that trajectories are necessarily  absorbed by a set of states that reside on the boundary of the positive orthant.  Hence, there is necessarily an irreversible `extinction' event in the system.  One immediate corollary to this is that the models admit no stationary distributions with mass near the equilibrium of the  deterministic model. Our results therefore demonstrate fundamentally different long-term dynamics than those observed in the corresponding deterministic models. Stochastic modeling of chemical reactions is particularly relevant in models of intracellular dynamics because critical proteins may have a low copy number per cell.

Depending upon the total molecular abundances of the constituent species, it may be that such an extinction is a rare event on the relevant time-scales of the system.   In this case, and under the assumptions used throughout this work,
the process will very likely seem to settle down to an equilibrium distribution long before the resulting instability will appear. This  distribution is called a \textit{quasi-stationary distribution}, and ACR-like results may still be obtained by consideration of this distribution.  In fact, in two examples provided here we observe that the quasi-stationary distribution of the absolutely robust species limits in a natural way to a Poisson distribution with mean value given by  the concentration predicted by the deterministic model.

The next section provides a motivating example for our main results.   In Section \ref{sec:ACR}, we formally introduce the concept of a species exhibiting ACR in the deterministic modeling context, and present both the main theoretical result from \cite{Sh-F1} together with  the main result being introduced here, Theorem \ref{thm:main}.  A detailed proof of Theorem \ref{thm:main}, and more general results, can be found in the Supplementary Material.  In Section \ref{sec:quasi} we consider the quasi-stationary distributions for the example models considered here.  We close with a brief discussion.

\section{A motivating example}
\label{sec:motivating_example}

  Consider the two-species activation/deactivation network
\begin{align}\label{eq:simple_model}
\begin{split}
\mathcal{R}_1: \; \; \; \;	A + B &\overset{\alpha}\to 2B\\
\mathcal{R}_2: \; \; \; \; \; \; \; \; \; \; \; 	B &\overset{\beta}{\to} A,
\end{split}
\end{align}
where $A$ is the active form of a protein, $B$ is the inactive form, and $\alpha$ and $\beta$ are positive rate constants \cite{Sh-F1}.  Notice that the inactive form $B$ regulates both the activation and deactivation steps of the mechanism. The usual differential equations governing the time evolution of the molar concentrations of $A$ and $B$, denoted $c_A$ and $c_B$ here, are
\begin{align}\label{eq:simple_de}
\begin{split}
	\dot c_A(t) &= -\alpha c_A(t) c_B(t) + \beta c_B(t)\\
	\dot c_B(t) &= \hspace{.1in}\alpha c_A(t)c_B(t) - \beta c_B(t).
\end{split}
\end{align}
Setting the left hand sides of the above equations to zero and solving yields the following values for the equilibrium concentrations:
\begin{align}\label{eq:simple_robust}
\begin{split}
	\bar{c}_A &= \beta/\alpha, \quad 
	\bar{c}_B = M - \beta/\alpha,
	\end{split}
\end{align}
\noindent where $M:=c_A(0)+c_B(0)$ is the total conserved protein concentration.   Equation \eqref{eq:simple_robust} shows that the deterministically modeled system \eqref{eq:simple_model} has ACR for the protein $A$, since all positive equilibria, no matter the initial condition, must satisfy $\bar{c}_A = \beta/\alpha$.    
It can also be easily checked that these equilibria are stable. The self-regulation of the mechanism \eqref{eq:simple_model}, although quite simple, predicts the remarkable property that the concentration of the active protein is kept within  tight bounds regardless of the value of the total protein concentration $M$.

Now consider the usual stochastic model for the network \eqref{eq:simple_model}, which treats the system as a continuous time Markov chain (CTMC) (see Fig. \ref{figure:CTMC_SIS_StateSpace}). Let $\mathbf{X}_A(t)$ and $\mathbf{X}_B(t)$ denote the individual counts of $A$ and $B$, respectively, and model the reactions as discrete events which occur stochastically in time. Under  usual assumptions on the rates, or propensities, of  the reactions, the first reaction can only occur if $\mathbf{X}_A(t) > 0$ and $\mathbf{X}_B(t) > 0$, and the second only if $\mathbf{X}_B(t) > 0$, which implies no reaction may proceed if $B$ is depleted completely.  
For a more thorough introduction to the stochastic models for biochemical systems see the Supplemental Material or \cite{AndKurtz20111}.

It is not hard to see that the long-term behavior of this CTMC is different from that of the deterministic model \eqref{eq:simple_de}. 
More specifically, it is possible for all of the inactive molecule to become active through the self-activation reaction $B \to A$. This sends the chain to the state
\begin{align}\label{eq:simple_stochastic}
\begin{split}
	\bar{\mathbf{X}}_A &= M, \qquad 
	\bar{\mathbf{X}}_B = 0,
	\end{split}
\end{align}
 where $M:=\mathbf{X}_A(0)+\mathbf{X}_B(0)$. After this time, neither reaction may occur and so no active molecules $A$ may be deactivated. Therefore, rather than having trajectories of the system spend most of their time near the value \eqref{eq:simple_robust}, over a sufficiently long time frame the inevitable outcome of the system is convergence to \eqref{eq:simple_stochastic}.

\begin{figure}
\centering
\fcolorbox{black}{white}{\includegraphics[width=.55\textwidth,height=0.19\textwidth]{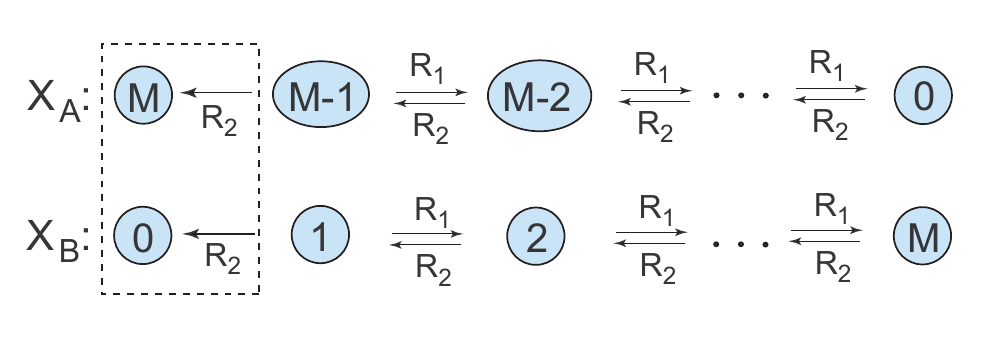}}
\caption{State space for the CTMC corresponding to \eqref{eq:simple_model}. The chain $\mathbf{X}(t) = (\mathbf{X}_A(t),\mathbf{X}_B(t))$ has a unique absorbing state at $(\bar{X}_A,\bar{X}_B) = (M,0)$ (boxed), where $M:=\mathbf{X}_A(0)+\mathbf{X}_B(0)$.}
\label{figure:CTMC_SIS_StateSpace}
\end{figure}

This  disparity between the long-term predictions of the deterministic mass-action model and that of the stochastic model at first seems to be at odds with the established result that the stochastic chain $\mathbf{X}(t) = (\mathbf{X}_1(t),\mathbf{X}_2(t),\ldots,\mathbf{X}_m(t))$ is 
well approximated by the corresponding deterministic model when molecular counts are high.  However, such results are valid only on \textit{finite time intervals}, and therefore stand silent on the long-term behavior of the models \cite{Kurtz21,Kurtz31}. Isolated examples of models exhibiting such a fundamental difference between the long-term behavior of the corresponding deterministic and stochastic models  
are well-known in the biochemical literature \cite{Allen1,Keizer1,O-S-W1,V-Q1}.

\section{Results Pertaining to Absolute Concentration Robustness}
\label{sec:ACR}

The main problem we consider here is: what structural conditions on the reaction network yield an absorption event similar to that of \eqref{eq:simple_model} for the corresponding stochastic system?
We are particularly interested in networks for which the deterministic model predicts ACR, and in this context we will use the original theorem due to Shinar and Feinberg \cite{Sh-F1}. 
We briefly introduce some terminology from chemical reaction network theory, including the network parameters $n$, $\ell$, and $s$.

We let $n$ denote the number of vertices of the reaction network.  These vertices are the linear combinations of the species at either end of a reaction arrow and are called \textit{complexes} in the chemical reaction network literature.
We will use this word, although the reference \cite{Sh-F1}  instead uses `nodes' to avoid confusion with the biological meaning of the word `complex'. 
Note that we may naturally  associate a complex with a non-negative vector $y$ in which the $j$th component of $y$ is the multiplicity of species $j$ in that complex.  For example, the complex $A + B$ in system \eqref{eq:simple_model} has associated vector $(1,1)$, whereas the complex $2B$ has associated vector $(0,2)$.

We let $\ell$ denote the number of connected components of the reaction network and associate to every reaction a {\it reaction vector} which determines the counts of the molecules gained and lost in one instance of that reaction. For example, for the reaction $y \to y'$, the reaction vector is $y' - y$, where we have slightly abused notation by writing the vector associated with a complex in the place of the complex itself.  We denote by $s$ the dimension of the span of all of the reaction vectors. For the activation/deactivation network \eqref{eq:simple_model} there are four complexes, $\{A+B, 2B, B, A\}$, and two connected components, $\{ A+B, 2B \}$ and $\{ B, A \}$. It follows that $n=4$ and $\ell = 2$.   We also notice that the reaction $A+B \to 2B$ has associated reaction vector
 $(-1,1)$, since the system loses one $A$ molecule and gains a net of one $B$ molecule due to one instance of the reaction. The reaction vector for the reaction $B \to A$ is $(1,-1)$ so that $s = 1$.

The \textit{deficiency} of a network is defined to be $\delta := n - \ell - s$.  The deficiency is known to only take non-negative values and has been utilized to show a variety of steady state results for mass-action systems, both deterministic and stochastic \cite{A-GAC-ONE1,A-C-K1,C-D-S-S1,F11,Fe21,Fe41,F21,Fe31}. For the network \eqref{eq:simple_model} we have
\[
	\delta = n - \ell - s = 4 - 2 - 1 = 1,
\]
so the deficiency is one.

We say two complexes are \textit{strongly linked} if there is a directed path of reactions from the first to the second, and also a directed path from the second back to the first.  A strong linkage class of a reaction network is a maximal subset of complexes that are strongly linked to each other.  A strong linkage class is furthermore called {\it terminal} if no complex in the class reacts to
a complex in another strong linkage class.  Complexes that do not belong to a terminal strong linkage class are called \textit{non-terminal}. For example, in the network \eqref{eq:simple_model} the complexes $A + B$ and $B$ are non-terminal, whereas the complexes $2B$ and $A$ are terminal. 
See also the network in Fig. \ref{figure:EnvZ/OmpR_mechanism}, where the terminal and non-terminal complexes are labeled in different colors.  Finally, we say that two complexes `differ only in species $S$' if the difference between them is a nonzero multiple of a single species $S$.  For example, in \eqref{eq:simple_model} the complexes $A + B$ and $B$ are both non-terminal and also differ only in species $A$.  

We can now state the main theorem of \cite{Sh-F1}.

\begin{theorem} \label{thm:marty}
	Consider a deterministic mass-action system that admits a positive steady state and suppose that the deficiency of the underlying reaction network is one.  If, in the network, there are two non-terminal complexes that differ only in species $S$, then the system has absolute concentration robustness  in $S$.
\end{theorem}
\noindent We have already seen the system \eqref{eq:simple_model} has a positive steady state if $M > \beta/\alpha$, that the underlying network has a deficiency of one, and that the non-terminal complexes $A+B$ and $B$ differ only in the species $A$. Thus, Theorem \ref{thm:marty} could be used to guarantee that the system exhibits absolute concentration robustness in $A$ even if the equilibrium could not have been calculated explicitly.

Note that Theorem \ref{thm:marty} stands silent on whether or not the equilibria of the deterministically modeled systems  are stable, either locally or globally.   In fact, there do exist ACR models for which the set of equilibria is unstable.  For example,
\begin{align*}
	A + B \overset{k_1}\to 2B,\quad 2A + B \overset{k_2}\to 3A,
\end{align*}
satisfies the requirements of Theorem \ref{thm:marty}, but the equilibrium $\bar c_A=k_1/k_2$ is  \textit{unstable.} The equilibria for the models considered in this article, however, are known to be globally stable.

\subsection{Stochastic Differences in Robustness} 

 The network \eqref{eq:simple_model} demonstrates that biochemical reaction networks satisfying the hypotheses of Theorem \ref{thm:marty}, and in which the ACR equilibria are known to be stable, may fail to exhibit similar stability when modeled stochastically. This disparity in the long-term dynamics is not restricted to only a few systems. In this section, we provide a theorem proved in the Supplemental Material which demonstrates that 
a large class of biochemical reaction networks satisfying the assumptions of Theorem \ref{thm:marty}, and thereby exhibit ACR when modeled  deterministically, have absorbing boundary states.  
In that case, when the deterministic ACR equilibria are stable, the stochastic model exhibits fundamentally different long-term dynamics.

In order to state the  result corresponding to Theorem \ref{thm:marty} for stochastically modeled systems we need a few more basic definitions. A chemical reaction network is said to be \textit{conservative} if there is a vector $w$ with strictly positive components for which $w \cdot (y' - y) = 0$ for all reactions $y \to y'$.  Note that in this case there is necessarily a conserved quantity $M>0$ so that for the given species set $\left\{ S_1, S_2, \ldots, S_m \right\}$ we have that
\[
	M:= w_{1} c_{1}(t)+ w_{2} c_{2}(t) + \cdots + w_{m} c_{m}(t)
\]
is invariant to the dynamics of the system.   For example, we have already seen that $M:=c_{A}(t) + c_{B}(t)$ is a conserved quantity for the system \eqref{eq:simple_model}.

Enumerating the reactions arbitrarily, we denote by $\lambda_k(x)$ the rate, or propensity, function of the reaction $y_k \to y_k'$, and note that it is reasonable to assume that $\lambda_k(x) >0$ if and only if $x_i\ge y_{ki}$, which says a reaction may only occur if there are a sufficient number of molecules in the system.   We say any family of rate functions satisfying this simple condition is \textit{stoichiometrically admissible}.   For example, the rate function $\lambda_k(x) = x_B$ would be stoichiometrically admissible for the reaction $B \to A$, but not for the reaction $A + B \to 2B$.  
A common choice for the propensities $\lambda_k(x)$ is {\it stochastic mass-action}
\begin{equation*}
\lambda_k(x) = \left\{ \begin{array}{ll} \displaystyle{\frac{k_k}{V^{|y_k| - 1}} \prod_{j=1}^m \binom{x_j}{y_{kj}}}
\; \; \; & \mbox{if } x_j-y_{kj} \geq 0 \mbox{ for all } j \\ 0, & \mbox{otherwise} \end{array} \right.
\end{equation*}
where $|y_k| = \sum_{j=1}^m y_{kj}$ and $V$ is the volume of the reaction vessel.  Note that under the assumption of mass-action kinetics, the propensity function is proportional to the number of ways in which one can choose the molecules necessary for the reaction to occur. 
Also, note that stochastic mass-action kinetics is stoichiometrically admissible.

We say a complex $y_k$ is \textit{turned off} at a particular state value $x$ if $x_i < y_{ki}$ for at least one $i$;
 otherwise we say the complex is \textit{turned on}.  Note that $\lambda_k(x) = 0$ for all $x$ at which $y_k$ is turned off.
  Next, we say that a complex $y$ is \textit{dominated} by the complex $y'$ (denoted $y \ll y'$) if  $y_i' \leq y_i$ for all $i=1, \ldots, m$.
   In particular, whenever two complexes differ in only one species, one necessarily dominates the other.  The intuition here is that if $y'$ is ever turned off,
   then so is $y$. 
  
Finally, we remind the reader that a state $\mathbf{X}$ of a CTMC is \textit{recurrent} if the chain satisfying 
$\mathbf{X}(0)  = \mathbf{X}$ returns to $\mathbf{X}$  with probability one, and that a recurrent state is \textit{positive recurrent} if the expected value of the return time
is finite.  It is also a basic fact that stationary distributions  only give mass to \textit{positive recurrent} states, showing that the long-term dynamics are restricted to those states.

The following is the main theoretical result of the paper and should be compared with Theorem \ref{thm:marty}. It, along with more general results allowing for \textit{higher deficiency}, is proved in the Supplemental Material.  
 
\begin{theorem}
\label{thm:main}
	Consider a reaction network which is conservative, has a deficiency of one, and for which the deterministically modeled mass-action system admits a positive equilibrium for some choice of rate constants.  Suppose that, in the network, there are two non-terminal complexes, $y_1$ and $y_2$ say, for which $y_1 \ll y_2$.  Then, for any choice of stoichiometrically admissible kinetics, 
	all non-terminal complexes of the network are turned off 
	at each positive recurrent state of the stochastically modeled system.
\end{theorem}
Hence, a trajectory of the stochastically modeled system will, with a probability of one, be absorbed by a set of states for which all of the non-terminal complexes are turned off.   In particular, the propensity of reactions out of the complexes $y_1$ and $y_2$ will be zero.

For instance, in the simple system \eqref{eq:simple_model} we have $A+B \ll A$, and Theorem~\ref{thm:main} states that the complexes $A+B$ and $B$ are turned off at any positive recurrent state.  This is easily verified since the only positive recurrent state is given by \eqref{eq:simple_stochastic}.
  In the supplementary material, we also provide a theorem characterizing the limiting behavior of the system \textit{after} the absorption event. In particular, we show that the reduced network is weakly reversible and has a deficiency of zero, and so admits a stationary distribution which is a product of Poissons \cite{A-C-K1}.   We also provide in the Supplementary Material an example in which the conclusions of the theorem do not hold if only the conservation requirement is dropped.

\begin{figure*}[t]
	\centering
	\begin{subfigure}[b]{0.4\textwidth}
	\centering
	\includegraphics[width=0.98\textwidth]{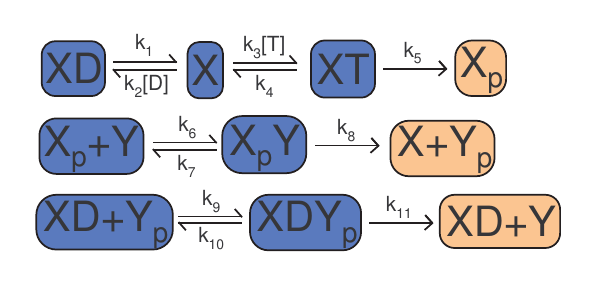}
	\caption{EnvZ/OmpR network}
	\label{figure1}
	\end{subfigure}
	\qquad \qquad
	\begin{subtable}[b]{0.4\textwidth}
	\centering
	\rowcolors{1}{}{lightgray}
	\footnotesize{
\begin{tabular}{|rl|rl|}
		\hline
		$k_1$ & $0.5$ $s^{-1}$ & $\kappa_1$ & $0.5$\\
		$k_2$ & $0.5$ $s^{-1}$ & $\kappa_2$ & $0.5$\\
		$k_3$ & $0.5$ $s^{-1}$ & $\kappa_3$ & $0.5$\\
		$k_4$ & $0.5$ $s^{-1}$ & $\kappa_4$ & $0.5$\\
		$k_5$ & $0.1$ $s^{-1}$ & $\kappa_5$ & $0.1$\\
		$k_6$ & $0.5$ $\mu M^{-1} s^{-1}$ & $\kappa_6$ & $0.02$\\
		$k_7$ & $0.5$ $s^{-1}$ & $\kappa_7$ & $0.5$\\
		$k_8$ & $0.5$ $s^{-1}$ & $\kappa_8$ & $0.5$\\
		$k_9$ & $0.5$ $\mu M^{-1} s^{-1}$ & $\kappa_9$ & $0.02$\\
		$k_{10}$ & $0.5$ $s^{-1}$ & $\kappa_{10}$ & $0.5$\\
		$k_{11}$ & $0.1$ $s^{-1}$ & $\kappa_{11}$ & $0.1$\\
		\hline
		$[D]$ & $1$ $\mu M$ & $n_A$ & $6.022 \times 10^{23}$ \\
		$[T]$ & $1$ $\mu M$ & $V$ & $4.151 \times 10^{-17}$ $L$ 	\\
		\hline
	\end{tabular} 
	}
	\caption{Parameter values}
	\end{subtable}
\caption{Hypothetical mechanism for the EnvZ/OmpR signal transduction system in \textit{Escherichia coli}. The mechanism is represented in (a) where the terminal (light orange) and nonterminal (dark blue) complexes are labeled. The parameter values used for numerical simulations are given in (b).}
\label{figure:EnvZ/OmpR_mechanism}
\end{figure*} 

\subsection{EnvZ/OmpR Signaling System}

In order to demonstrate Theorem \ref{thm:main}   on a more complicated example, we now consider a model of the two-component EnvZ/OmpR signaling system in \textit{Escherichia coli} \cite{B-G1}, which was also studied in \cite{Sh-F1}. The histidine kinase EnvZ is sensitive to extracellular osmolarity, the input of the system, and in its active phosphorylated form EnvZ-P is able to phosphorylate the response regulator OmpR into the active form OmpR-P. OmpR-P in turn signals the transcription of the porins OmpF and OmpC. The level of OmpR-P can therefore be thought of as the output of the system. EnvZ is also known to play a role in the regulation of the level of OmpR-P through dephosphorylation \cite{Zhu20001}.

Multiple mechanisms have been proposed for the EnvZ/OmpR system \cite{B-G1,Igoshin20081,Shinar20071}. We will consider the mechanism given in Fig. \ref{figure:EnvZ/OmpR_mechanism} which was proposed in \cite{Sh-F1,Shinar20071}. Here, it is imagined that ADP ($D$) and ATP ($T$) interact with EnvZ ($X$) to produce bound complexes, but that only ATP can successfully transfer the phosphate group ($P$) to EnvZ to form EnvZ-P ($X_p$). EnvZ-P may then transfer the phosphate group to OmpR ($Y$) to form OmpR-P ($Y_p$) while the modified EnvZ-ADP complex regulates the dephophorylation of OmpR-P. Assuming $[D]$, $[T]$ and $[P]$ are of sufficient quantity to be relatively unchanged by the course of the reaction, we may incorporate them into the rate constants, yielding the network contained in Fig. \ref{figure:EnvZ/OmpR_mechanism}.

Notice that the nonterminal complexes $XD$ and $XD+Y_p$ differ only in the species $Y_p$. Since the network also has deficiency one (see Supplemental Material), by Theorem \ref{thm:marty} we may conclude that the deterministically modeled mass-action system exhibits ACR in $Y_p$. It is shown in the Supplemental Materials of \cite{Sh-F1} that the ACR value is
\begin{equation}
\label{cyp}
Y_p = \frac{k_1k_3k_5(k_{10}+k_{11})[T]}{k_2(k_4+k_5)k_9k_{11}[D]}.
\end{equation}
The corresponding equilibrium is stable and, consequently, the deterministic model predicts that the active form of the response regulator, OmpR-P, is robust to the overall level of the signaling proteins EnvZ and OmpR. That is to say, the prediction is that the system will exhibit a similar response regardless of differences in these internal characteristics.

We now consider the stochastic model for the network in Fig. \ref{figure:EnvZ/OmpR_mechanism}.   It can be  seen that, in addition to satisfying the assumptions of Theorem \ref{thm:marty}, the network has the conservation relations
\begin{equation}
\label{eqn:conservation}
\begin{split}
X_{tot} & := X+XD+XT+X_p+X_pY+XDY_p\\
Y_{tot} & := Y+X_pY+XDY_p+Y_p.
\end{split}
\end{equation}
Since the sum of these two relations has support on all species, it follows that the network is conservative.
By Theorem \ref{thm:main}, the stochastic model converges in finite time to a state (or set of states) for which all non-terminal complexes are turned off. 

Note that, since the species $X$, $XD$, $XT$, $X_pY$ and $XDY_p$ are also non-terminal complexes, Theorem \ref{thm:main} guarantees that each will be zero after the inevitable absorption event.  Since $X_{tot}$ is conserved, we may also conclude that $X_p = X_{tot}$ at this time.  Finally, as $X_p+ Y$ is also non-terminal, we may also conclude that $Y=0$.
 In this way, the unique sink of the stochastic chain is seen to be
\begin{equation}
\label{stochastic_applied}
\begin{split}
Y_p &=Y_{tot}\\
X_p &=X_{tot}\\
X   &=XD=XT=X_pY=XDY_p = Y=0.
\end{split}
\end{equation}

\section{Time until absorption and quasi-stationary distributions}
\label{sec:quasi}

A disparity between the long-term behavior of deterministic and stochastic models of some reaction networks is well-known in the literature \cite{Allen1,Keizer1,O-S-W1,V-Q1}. The activation/deactivation network \eqref{eq:simple_model} is a canonical example 
that mimics the well-studied stochastic susceptible-infected-susceptible (SIS) epidemic model \cite{K-L1,Nasell19961,Nasell19991,W-D1}, where $A$ and $B$ correspond to the number of healthy and infected individuals, respectively.
In that setting,  
the state \eqref{eq:simple_stochastic} corresponds to a true ``extinction state.'' A similar extinction effect is also achieved in many population biology models where stochastic effects may irreversibly drive a population to zero \cite{Meleard_Quasi1}.

Although the chain associated with \eqref{eq:simple_model} will inevitably converge to the extinction state \eqref{eq:simple_stochastic}, such an event may be exceedingly rare on  biologically reasonable timescales. 
In such situations,  the absorbing state is not of practical concern, and, in fact, the process may seem to settle to a stationary distribution. This distribution is called the \textit{quasi-stationary probability distribution}, and it is useful in analyzing the transient behavior of a model before the absorption event.  
We suspect that stable ACR-like behavior may still be attained in the present context by consideration of this distribution.

Suppose we denote the absorbing states of a process $\mathbf{X}(t)$ by $\partial A$.  If we define  $p_{\mathbf{x}}(t) = P \left\{ \mathbf{X}(t) = {\mathbf{x}} \right\}$, where ${\mathbf{x}}$ is an arbitrary state in the state space, and $P_{\partial A}(t) = \sum_{{\mathbf{x}} \in \partial A} p_{\mathbf{x}}(t)$, then the  transition probabilities conditioned upon non-extinction, $q_{\mathbf{x}}(t)$, are given for $t\ge 0$  by
\begin{equation}
\label{quasi-stationary}
q_{\mathbf{x}}(t) = P \left\{ \mathbf{X}(t) = {\mathbf{x}} \; | \; \mathbf{X}(t) \notin \partial A\right\} = \frac{p_{\mathbf{x}}(t)}{1 - p_{\partial A}(t)}.
\end{equation}
\noindent 
The limiting vector $\pi:=\lim_{t\to \infty} q(t)$, if it exists, is the \textit{quasi-limiting} distribution of the process, which equals the \textit{quasi-stationary} distribution in the present context.    See \cite{D-S1} for proofs of the facts that such a distribution exists, and is unique, in the setting of our Theorem \ref{thm:main}.   See \cite{Meleard_Quasi1} for a recent survey article on quasi-stationary distributions in the context of population processes.

\begin{figure}
	 \centering
        \includegraphics[width=0.7\textwidth
        ]{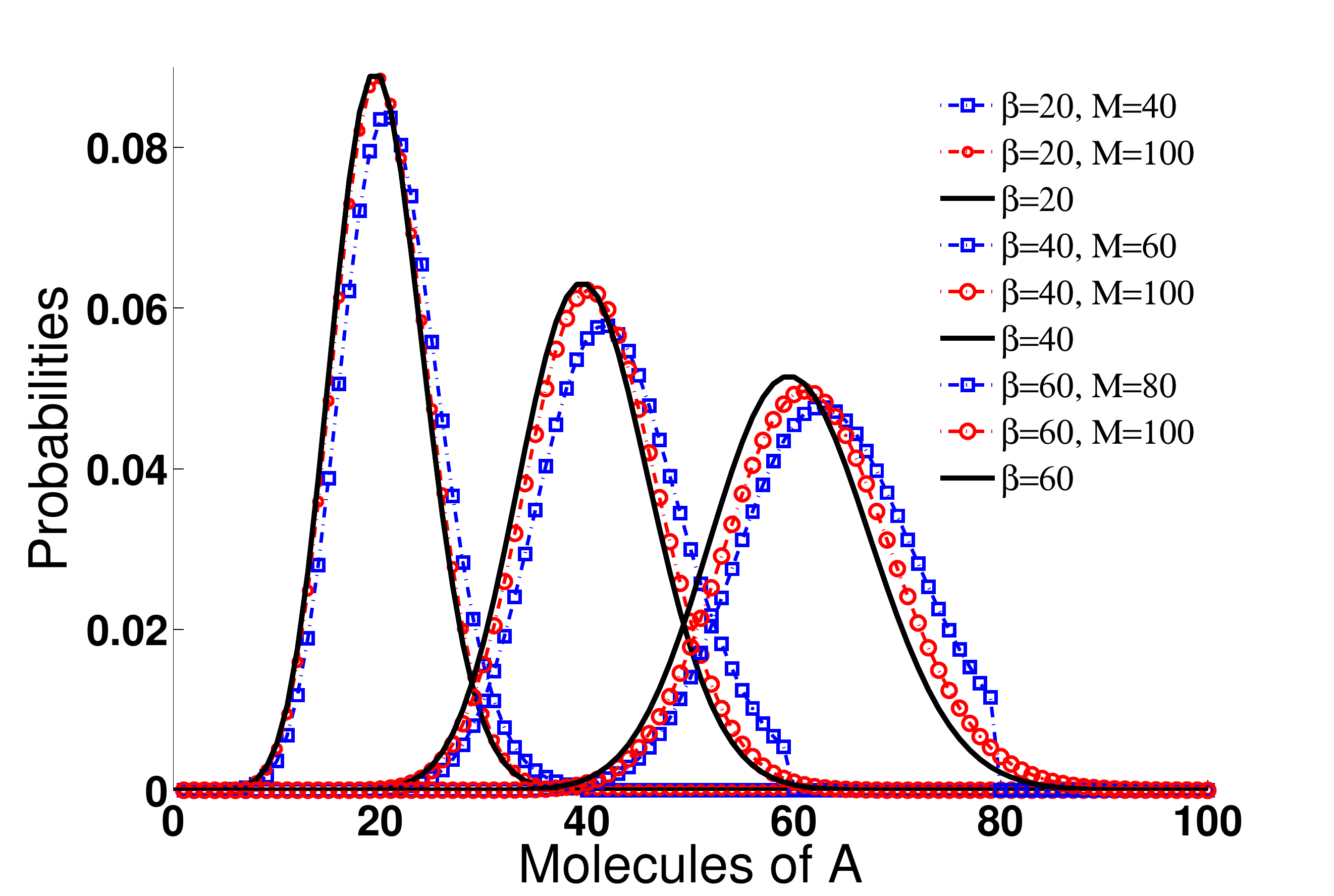}
        \caption{Quasi-stationary distributions of $\mathbf{X}_A$ with $\alpha = 1$ and various values of $\beta$ and $M$.   As $M \to \infty$ the quasi-stationary distributions approach the overlain Poisson distributions \eqref{poisson} (solid line). The iterative procedure of \cite{Cavender19781,K-L1,Nasell19961} was used to construct the plots.}
        \label{figure:quasi-stationary_SIS}
\end{figure}

Reconsider the activation/deactivation system \eqref{eq:simple_model}.  We begin by reimagining the chain $\mathbf{X}(t) = (\mathbf{X}_A(t),\mathbf{X}_B(t))$ as a birth-death process following $\mathbf{X}_B(t)$ with the extinction state $\mathbf{X}_B =0$ corresponding to \eqref{eq:simple_stochastic}. The first reaction corresponds to a ``birth'' since $B$ is increased by one while the second reaction corresponds to a ``death'' since $B$ is decreased by one. The chain for  $\mathbf{X}_A$ can then be determined by the conservation 
  relation $\mathbf{X}_A(t) = M - \mathbf{X}_B(t)$. Under mass-action kinetics, the corresponding birth and death propensities, $\lambda(i)$ and $\mu(i)$, respectively, are given by 
\begin{equation}
\label{propensities}
\begin{split}
\lambda(i) & =  \alpha i ( M - i), \\
\mu(i) &  = \beta i,
\end{split}
\end{equation}
where $i = 0, \ldots, M$ corresponds to the state with $i$ molecules of $B$. Notice that $\lambda(0)=\lambda(M)=\mu(0) =0$ so that $\mathbf{X}_B = 0$ is an absorbing state and $\mathbf{X}_B = M$ is a reflecting state. Up to rescaling of the rate constants, this chain is identical to the stochastic SIS epidemic model considered in \cite{K-L1,Nasell19961,Nasell19991}.

We are interested in whether the quasi-stationary distribution displays a form of robustness in $\mathbf{X}_A$ with respect to changes in overall molecularity $M$.  Furthermore, we are interested in the case of large $M$, since that is when the time to extinction is large.  
 In the Supplemental Material we prove that as $M\to \infty$, the quasi-stationary distribution for $\mathbf{X}_A$ approaches the Poisson distribution
 \begin{equation}
\pi(i) = \frac{e^{-\left( \frac{\beta}{\alpha} \right)}}{i!} \left( \frac{\beta}{\alpha} \right)^i,
\label{poisson}
\end{equation}
which has previously been shown to be the limit of one commonly used approximation to the quasi-stationary distribution \cite{H-C-L-H1}.  In Fig. \ref{figure:quasi-stationary_SIS}, 
we graphically demonstrate the convergence by providing realizations of the quasi-stationary distribution for different values of $\beta$ and $M$.

Now consider the EnvZ/OmpR signaling system from Fig. \ref{figure:EnvZ/OmpR_mechanism}.
As in the simple activation/deactivation network, the expected time before entering the state \eqref{stochastic_applied} may be very large. We therefore consider the quasi-stationary distribution of the process.  
This distribution cannot be computed using a simple analytic formula, so we approximate it numerically
and point the reader to the Supplementary Material for the details of our computational methods.  
We use the parameter values given in the table in Fig. \ref{figure:EnvZ/OmpR_mechanism}(b), which have been chosen to be in close agreement with the values in \cite{Igoshin20081}.
 It was found in \cite{Cai20021} that a typical {\it E. coli} cell has roughly $100$ total molecules of EnvZ and $3500$ molecules of OmpR. Therefore, for our simulations we chose a ratio of $Y_{tot}:X_{tot}$ of $35:1$. Based on the deterministic model, the anticipated mean of the ACR species $Y_p$ in the quasi-stationary distribution is $25$.

The results of the simulations are shown in Fig. \ref{figure:Yp}. In Fig. \ref{figure:Yp}(a) we see that, for low molecularity of EnvZ and OmpR ($1$ and $35$, respectively), the chains converge to the boundary state \eqref{stochastic_applied}, even while the quasi-stationary distribution becomes apparent. 
In Fig. \ref{figure:Yp}(b) we see that the quasi-stationary distribution of $Y_p$ appears to converge to a Poisson distribution centered around the deterministic steady state as the total molecularity grows. 



\begin{figure*}
   \centering
	\begin{subfigure}[b]{0.47\textwidth}
	\centering
	\includegraphics[width=0.99\textwidth]{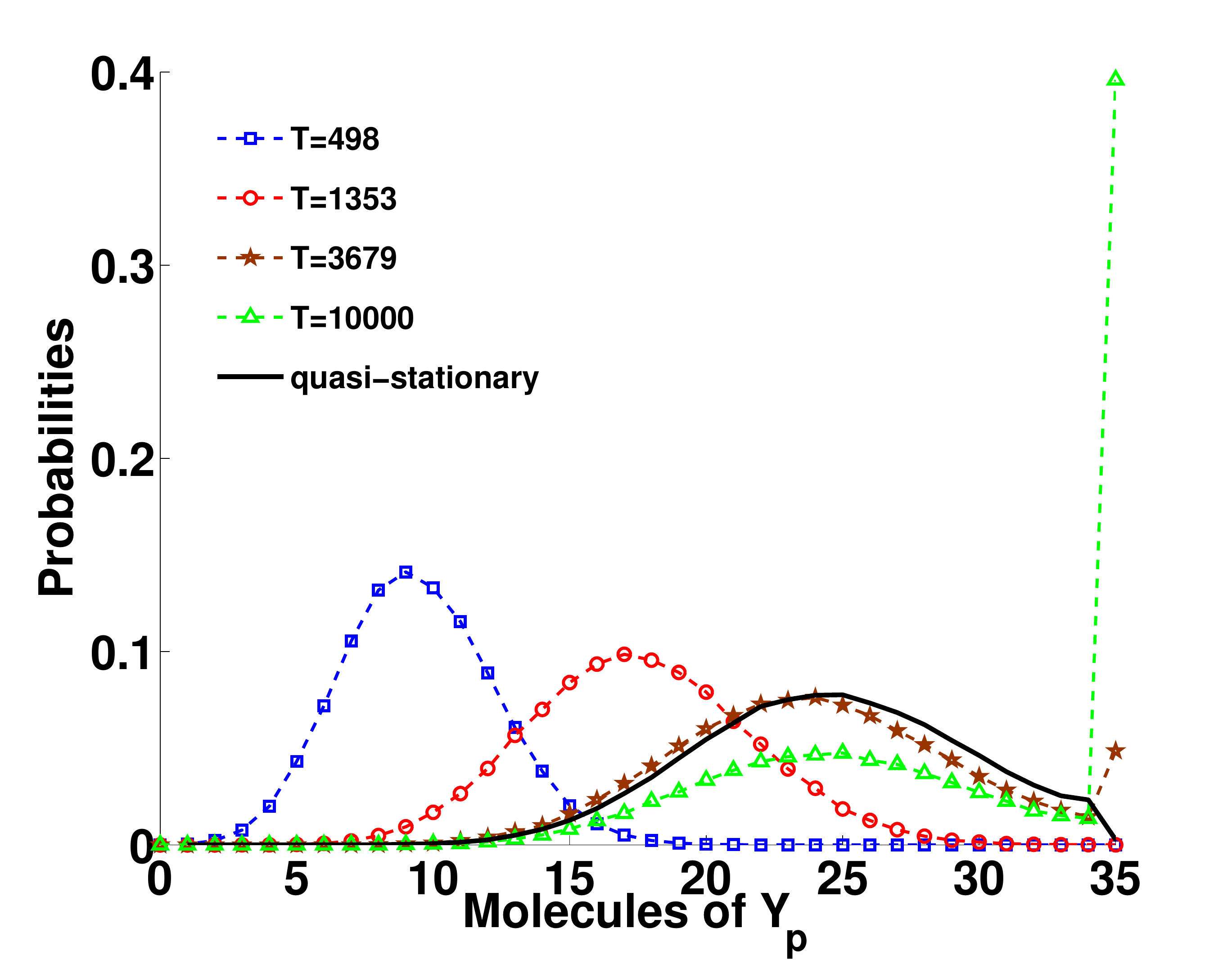}
	\caption{Probability profiles over time for $X_{tot}=1$, and $Y_{tot}=35$.}
	\label{figure:gillespie}
	\end{subfigure}
	\begin{subfigure}[b]{.47\textwidth}
	\centering
	\includegraphics[width=0.99\textwidth]{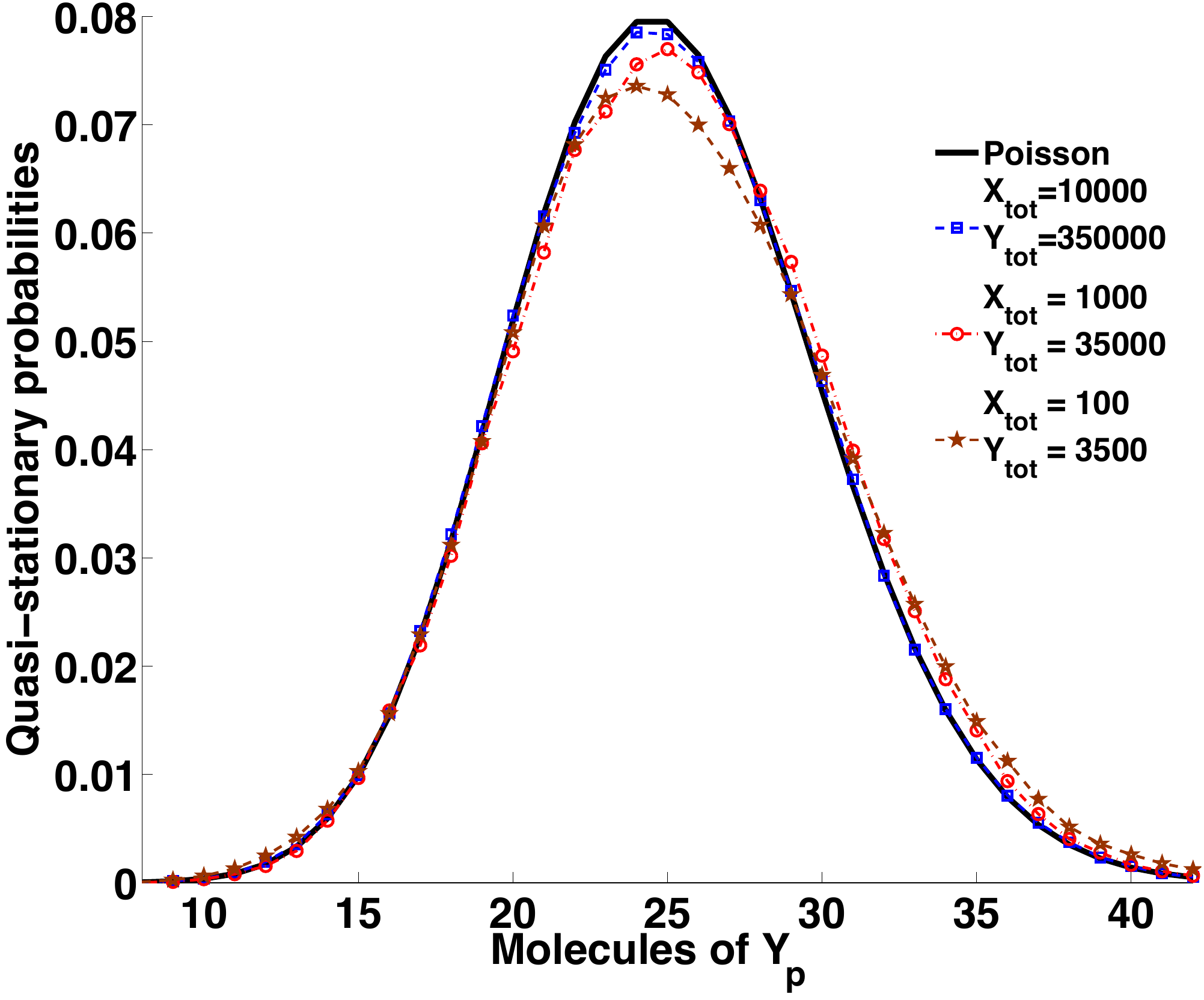}
	\caption{Quasi-stationary distributions of $Y_p$ for various $X_{tot}$, $Y_{tot}$.}
	\label{figure:quasi}
	\end{subfigure}
	\caption{Results of numerical simulation for the EnvZ/OmpR mechanism of Fig. \ref{figure:EnvZ/OmpR_mechanism}. In (a), we display the approximate probability distribution of the chain $\mathbf{X}(t)$ at four time points, computed by averaging $10^5$ independent realizations of the process with $X_{tot}=1$ and $Y_{tot}=35$, and  initial conditions $X(0) = X_{tot}$ and $Y(0)=Y_{tot}$. The corresponding quasi-stationary distribution is overlain (dotted). In (b), we display approximations of the quasi-stationary distributions for various values of $Y_{tot}$ and $X_{tot}$ in the ratio $35:1$. Details of the simulations are provided in the Supplementary Material.  For comparison, a Poisson distribution with mean $25$ is overlain (black). Convergence to the Poisson distribution as the total molecularity increases is apparent.}
	\label{figure:Yp}
	\hspace*{\fill}
\end{figure*}

\section{Discussion}

Robustness and stability in the face of varying  environments is of fundamental importance to the proper functioning of many biological processes, and understanding this behavior is one arena where mathematics can play a role in elucidating biological phenomena.  In this paper, together with its Supplemental Material, we have outlined a class of structural conditions which are sufficient to guarantee convergence of trajectories of stochastically modeled systems to an absorbing boundary set. Notably, these conditions overlap significantly with a  set of structural conditions which are known to confer ACR on the corresponding deterministic models. For such ACR models with stable equilibria, our results present a significant disparity in the predictions for the limiting behavior of the systems.

This work highlights several points.  First,
it is 
surprising that the long-term behavior of  
such a large class of systems considered in \cite{Sh-F1} is fundamentally different for 
the stochastic and deterministic models.
Second, deterministic models are typically unable to capture trapping phenomena such as those described here unless they are artificially modified, for instance by adding degradation terms. We have shown here that for a wide range of models no such ad hoc modifications are necessary for stochastically modeled systems.  Third, in regions where the time to absorption for the stochastic model is large relative to the time-scale of the system, the proper object of study when considering the stochastic analogue of ACR behavior is the quasi-stationary distribution, as opposed to the stationary distribution.

This work suggests a number of promising avenues for future research.  First, finding more general conditions for which the conclusions of Theorem \ref{thm:main} hold will be a focus. 
 In particular, weakening the requirement that the system possesses a conservation relation will allow the results to be applicable to more models arising in ecology and population processes, which typically do not satisfy such an assumption, though do often possess low numbers of the constituent species \cite{Smith20111}. 
Second, the question of when equilibria exhibiting ACR in the deterministic modeling context are stable or unstable is, to the best of the authors' knowledge, currently open and has to date received surprisingly little consideration in the literature.
 Third, we have observed in the two examples considered here
the recurrence of the Poisson distribution when analyzing a certain limiting behavior of their quasi-stationary distributions. It is a suspicion of the authors that this phenomenon applies more generally to other systems satisfying ACR,
and this will be investigated.

Results of the type presented here are not only of theoretical interest.  In particular, they are exactly the types of results required in order to automate the multi-scale  reduction methods for stochastic models of biochemical processes currently being explored in the probability literature \cite{KK20131}.  For example, in order to determine  the behavior of species operating on a time-scale that is slower than other species, it is necessary to first understand the long-term dynamics of the species on the \textit{fast time-scale}  in order to perform either the necessary stochastic averaging, or to recognize that some species will have gone extinct.


This work is part of a growing research field in which mathematical methods are developed in order to rise above the bewildering complexity of biochemical processes.  Algebraic methods, for example, have been successfully employed in a number of areas related to the equilibria of mass-action systems, where the steady states of such systems form a real algebraic variety \cite{C-D-S-S1,M-Gunawardena1,M-D-S-C1}. We believe there is significant progress to be made towards the understanding of stochastically modeled systems through analysis of the underlying network, and we hope to uncover the important substructures hidden in the complexity of biochemical networks that inform system behavior, both on short and long time-frames.

\vspace{.2in}
\noindent\textbf{Acknowledgments.}
	Anderson was supported by NSF grants DMS-1009275 and DMS-1318832.   Johnston was supported by NSF grant DMS-1009275 and NIH grant R01-GM086881.  Enciso was supported by NSF grants DMS-1122478 and DMS-1129008.   We gratefully acknowledge the American  Institute of Mathematics (AIM) for hosting a workshop at which this research was initiated.



\pagebreak

\appendix
\begin{center}
\noindent {\large Supporting  Material for:\\ \vspace{0.1in} \textbf{Stochastic analysis of biochemical reaction networks with absolute concentration robustness}}
\end{center}

\section{Introduction}

In this supplementary material, we will, among other things, provide the proof of the main theorem in the article text (Theorem \ref{thm:cor1}).   We will also provide a 
 result, Theorem \ref{thm:main_general}, which gives more general conditions under which  an absorption event is guaranteed to occur.  In particular, we note that  Theorem \ref{thm:main_general} applies to models with deficiencies that are greater than one, whereas Theorem \ref{thm:cor1} does not.  In  Section \ref{sec:higher_deficiencies} we provide  an example of a network from the biology literature which has a deficiency of two for which Theorem \ref{thm:main_general} applies, whereas Theorem \ref{thm:cor1} stands silent.  In Section \ref{sec:further} we conjecture that the conclusions of Theorem \ref{thm:main_general} and Theorem \ref{thm:cor1} hold for the whole class of systems permitting absolute concentration robustness when modeled deterministically.

The outline of the remaining text is the following.  We will begin in Section \ref{sec:Background} with a synopsis of the required terminology and background material related to chemical reaction network theory, and both the deterministic and stochastic models used for the dynamics of biochemical systems.  In Section \ref{sec:Main_Results}, we state and prove our main results.  These results give conditions on the associated network for when an absorption event is guaranteed to occur when the system is modeled with stochastic dynamics.  
The conditions of our main results overlap significantly with the  conditions of Feinberg and Shinar that guarantee absolute concentration robustness for a species in the  deterministic modeling context \cite{Sh-F}. When such equilibria are {\it stable}, therefore, the results of this paper represent a distinction in the long-term behavior of a class of deterministically and stochastically modeled systems.

An immediate question that comes to mind in light of Theorems \ref{thm:main_general} and \ref{thm:cor1} is: what is the structure and long-term behavior of the post-absorption network? In Section \ref{sec:after}, we answer this question in detail for the class of systems considered by Theorem \ref{thm:cor1}. 
 Finally, in Section \ref{analysissection} we consider the time until the guaranteed absorption event takes place.   If the time to absorption is large relative to the relevant time-scales of the system, the processes will seem to settle down long before the resulting instability will appear.  This limiting distribution is called a \textit{quasi-stationary} distribution, and is also considered in Section \ref{analysissection} for the  network models studied here.

\section{Background, terminology, and notation}
\label{sec:Background}

In this section, we will briefly introduce  necessary terminology and notation. For a more complete introduction, the reader is directed to \cite{E-T,F3,Gunawardena} for background on {\it Chemical Reaction Network Theory} (CRNT), \cite{AndKurtz2011} for background on stochastic chemical reaction systems, and \cite{Allen,Lawler} for general information on {\it Continuous Time Markov Chains} (CTMC).

We will use the following two examples throughout this paper. We will use them to both illustrate the background concepts and as an application of our main result, Theorem \ref{thm:main_general}.

\vspace{0.2in}

\noindent {\it Example 1:} Consider the activation/deactivation network
\begin{equation}
\label{system1}
\begin{split}
A + B & \stackrel{\alpha}{\rightarrow} 2B \\
B & \stackrel{\beta}{\rightarrow} A
\end{split}
\end{equation}
where $A$ and $B$ denote the active and inactive form of a protein, respectively. We imagine here that the first reaction corresponds to deactivation of an active protein facilitated by an inactive protein, and that the second reaction corresponds to spontaneous activation of the inactive protein.

We have borrowed this network and interpretation from \cite{Sh-F}. We note that the same network is produced by reversing the roles of $A$ and $B$ to produce a network of facilitated activation and spontaneous deactivation.  It is also considered as a chemical reaction network in \cite{O-S-W}, as a mechanism for the spread of rumours in \cite{Bartholomew1976}, as a model of logistic population growth \cite{Norden1982}, and as a model for $SIS$ epidemic growth in \cite{Artalejo2010,K-L,Nasell1996,Nasell1999,W-D}.

\vspace{0.2in}

\noindent {\it Example 2:} Consider the hypothetical EnvZ/OmpR signal transmission network
\begin{equation}
\label{system2}
\begin{split}
& XD \mathop{\stackrel{k_1}{\rightleftarrows}}_{k_2[D]} X \mathop{\stackrel{k_3[T]}{\rightleftarrows}}_{k_4} XT \stackrel{k_5}{\rightarrow} X_p \\
& X_p + Y \mathop{\stackrel{k_6}{\rightleftarrows}}_{k_7} X_pY \stackrel{k_8}{\rightarrow} X + Y_p \\
& XD + Y_p \mathop{\stackrel{k_9}{\rightleftarrows}}_{k_{10}} XDY_p \stackrel{k_{11}}{\rightarrow} XD + Y
\end{split}
\end{equation}
where the species are $X$=EnvZ, $Y$=OmpR, $X_p$=EnvZ-P, $Y_p$=OmpR-P, $D$=ADP, and $T$=ATP. The species $D$ and $T$, which represent ADP and ATP, respectively, are assumed to be in sufficient quantity so that binding with $X$ to form $XD$ and $XT$ do not appreciably change their overall quantities. Their kinetic effects are therefore incorporated into the rate constants.

This model was proposed in \cite{Shinar2007} and \cite{Sh-F} to underlie the EnvZ/OmpR signaling system in {\it Escherichia coli} as studied experimentally in \cite{B-G}. The individual sequence of reactions are imagined to represent the phosphate (signal) transfer from the sensor molecule EnvZ to the response molecule OmpR. The first chain of reactions corresponds to the phosphorylization of EnvZ from the donor molecules ADP and ATP. Notably, only ATP is able to successfully phosphorylate EnvZ. The second chain corresponds to the transfer of the phosphate group from EnvZ to OmpR. The third chain corresponds to the dephosphorylization of OmpR-P by the ADP-EnvZ complex. It is notable that the signaling molecule EnvZ serves a dual purpose of phosphorylating OmpR and, in the modified ADP-EnvZ form, dephosphorylating OmpR-P.

\subsection{Chemical reaction networks}
\label{crnsection}

The following is the basic object of study in CRNT \cite{E-T,F3,Gunawardena,H-J1,Sh-F}.

\begin{definition}
\label{crn}
A {\it chemical reaction network} is a triple $\{\mathcal S, \mathcal C, \mathcal R\}$ of finite sets:
\begin{enumerate}
\item 
A {\it species set} $\mathcal{S} = \{ X_1, \ldots, X_m \}$ containing the basic species/molecules capable of undergoing chemical change.
\item
A {\it reaction set} $\mathcal{R} = \{\mathcal{R}_1, \ldots, \mathcal{R}_r \}$ containing the {\it elementary reactions}
\begin{equation}
\label{eqn:reaction}
\mathcal{R}_i: \; \; \; \; \; \; \sum_{j=1}^m y_{ij} X_j \; \longrightarrow \; \sum_{j=1}^m y'_{ij} X_j, \; \; \; \; \; \; \; i = 1, \ldots, r
\end{equation}
where the {\it stoichiometric coefficients} $y_{ij},y'_{ij} \in \mathbb{Z}_{\geq 0}$ keep track of the multiplicity of the individual species within the reactions. Using a slight abuse of notation, we allow the reactions in (\ref{eqn:reaction}) to be represented as $y_i \to y_i'$ where $y_i = (y_{i1},\ldots, y_{im})$ and $y_i' = (y'_{i1},\ldots,y'_{im})$.  We say that the reaction $y_i \to y_i'$ is a reaction {\it out} of complex $y_i$ and {\it into} the complex $y_i'$.
\item
A {\it complex set} $\mathcal{C}$ containing the linear combination of the species on the left-hand and right-hand sides of the reaction arrow in (\ref{eqn:reaction}). Using the same abuse of notation as before, we allow complexes to be represented by their support vectors so that $\mathcal{C} = \bigcup_{i=1}^r \{ y_i, y_i' \}$. The number of stoichiometrically distinct complexes is denoted by $n$ (i.e. $|\mathcal{C}| = n$).
\end{enumerate}
\end{definition}

\begin{remark}
It is typical within CRNT to assume that (i) every species appears in at least one complex, (ii) every complex appears in at least one reaction, and (iii) there are no self-reactions (i.e. reactions of the form $y_i \to y_i'$ where $y_i = y_i'$).
\end{remark}


To each reaction $y_i \to y_i'$, $i=1, \ldots, r$, we furthermore associate the {\it reaction vector} $y_i' - y_i \in \mathbb{Z}^m$. The components of the reaction vectors correspond to how many copies of each species are gained or lost by each instance of an individual reaction. The {\it stoichiometric subspace} of the network is defined to be $S = \mbox{span} \left\{ y_i' - y_i \; | \; i =1, \ldots, r \right\}$ and the dimension of $S$ is denoted $s = \mbox{dim}(S)$.

By representing each stoichiometrically distinct complex only once, the chemical reaction network $(\mathcal{S},\mathcal{C},\mathcal{R})$ can be interpreted as a directed graph $G(V,E)$ where the vertex set is given by $V = \mathcal{C}$ and the edge set is given $E = \mathcal{R}$. Following \cite{Sh-F}, we introduce the following connectivity relations.
\begin{enumerate}
\item
The complexes $y, y' \in \mathcal{C}$ are {\it directly linked} (denoted $y \leftrightarrow y'$) if $y \to y'$ or $y' \to y$;
\item
The complexes $y, y' \in \mathcal{C}$ are {\it linked} (denoted $y \sim  y'$) if either $(i)$ $y = y'$, or $(ii)$  there exists a sequence of complexes such that $y = y_{\mu(1)} \leftrightarrow y_{\mu(2)} \leftrightarrow \cdots \leftrightarrow y_{\mu(l)} = y'$;
\item
There is a {\it path} from $y \in \mathcal{C}$ to $y' \in \mathcal{C}$ (denoted $y \Rightarrow y'$) if there exists a sequence of complexes such that $y = y_{\mu(1)} \rightarrow y_{\mu(2)} \rightarrow \cdots \rightarrow y_{\mu(l)} = y'$; and
\item
The complexes $y, y' \in \mathcal{C}$ are {\it strongly linked} (denoted $y \approx y'$) if either $(i)$ $y =y'$, or $(ii)$  $y \Rightarrow y'$ and $y' \Rightarrow y$.
\end{enumerate}
 The connectivity relations ``$\sim$'' and ``$\approx$'' allow the complex set $\mathcal{C}$ to be partitioned into equivalence classes called {\it linkage classes} and {\it strong linkage classes}, respectively. That is to say, $y, y' \in \mathcal{C}$ are in the same linkage class $\mathcal{L}$ (respectively, strong linkage class $\Lambda$) if and only if $y \sim y'$ (respectively, $y \approx y'$). A strong linkage class $\Lambda$ is said to be {\it terminal} if there is no reaction $y \to y'$ where $y \in \Lambda$ but $y' \not\in \Lambda$. We will say that a complex is {\it non-terminal} if it is not contained in a terminal strong linkage class. We will furthermore say that two non-terminal complexes $y$ and $y'$ {\it differ in the species} $X_i$ if $y' = y + \alpha e_i$ where $\alpha > 0$, and $e_i \in \mathbb{R}^{m}$ is the canonical basis vector with a one in the $ith$ component and zeros elsewhere.

  The set of linkage classes is denoted $\mathcal{L} = \left\{ \mathcal{L}_1, \ldots, \mathcal{L}_{\ell} \right\}$ where $| \mathcal{L} | = \ell$ and the set of terminal strong linkage classes is denoted $\Lambda = \left\{ \Lambda_1, \ldots, \Lambda_t \right\}$ where $| \Lambda | = t$. A network is called {\it weakly reversible} if the linkage classes and strong linkage classes coincide (i.e. if $y \Rightarrow y'$ implies $y' \Rightarrow y$).

Finally, we introduce the following network parameter which has been the focus of much study in the chemical reaction network literature \cite{A-GAC-ONE,C-D-S-S,Fe2,Fe4,F1,H}.

\begin{definition}
\label{deficiency}
The {\it deficiency} of a chemical reaction network $(\mathcal{S},\mathcal{C},\mathcal{R})$ is given by $\delta = n - \ell - s$ where $n$ is the number of stoichiometric distinct complexes, $\ell$ is the number of linkage classes, and $s$ is the dimension of the stoichiometric subspace $S$.
\end{definition}
\noindent The deficiency of a chemical reaction network  does not depend on the choice of kinetics or, in the case of mass-action kinetics, on the choice of rate constants, and is known to only take non-negative integer values \cite{F3,Gunawardena}.  Nevertheless, many  dynamical properties for deterministically-modeled chemical reaction systems are characterized in terms  of the underlying network's deficiency  \cite{F1,F2,Fe2,Fe3,Fe4}. 

\vspace{0.2in}

\noindent {\it Example 1:} The network (\ref{system1}) has the species set $\mathcal{S} = \left\{ A, B \right\}$, the complex set $\mathcal{C} = \left\{ A + B, 2B, B, A \right\}$, and the reaction set $\mathcal{R} = \left\{ A + B \to 2B, B \to A \right\}$. The linkage classes are $\mathcal{L}_1 = \left\{ A + B, 2B \right\}$ and $\mathcal{L}_2 = \left\{ B, A \right\}$, and the strong   linkage classes are $\Lambda_1 = \left\{ A+B \right\}$, $\Lambda_2 = \left\{ 2B \right\}$, $\Lambda_3 = \left\{ B \right\}$ and $\Lambda_4 = \left\{ A \right\}$, of which $\Lambda_2$ and $\Lambda_4$ are terminal. The non-terminal complexes are $A+B$ and $B$, which differ in only the species $A$. The deficiency can easily be computed: $\delta = n - \ell - s = 4 - 2 - 1 = 1$.

\vspace{0.2in}

\noindent {\it Example 2:} The network (\ref{system2}) has the species set
\[\mathcal{S} = \left\{ X,Y,X_p,Y_p,XD,XT,X_pY,XDY_p \right\},\]
the complex set
\[\mathcal{C} = \left\{ XD,X,XT,X_p,X_p+Y,X_pY,X+Y_p,XD+Y_p,XDY_p,XD+Y \right\},\]
and the reaction set
\[\begin{split} \mathcal{R} = & \left\{ XD \to X, X \to XD, X \to XT, XT \to X, XT \to X_p, \right. \\ & \left. \; \;  X_p +Y \to X_p Y, X_p Y \to X_p + Y, X_p Y \to X + Y_p, \right. \\ & \left. \; \;  XD + Y_p \to XDY_p,XDY_p \to XD + Y_p, XDY_p \to XD + Y \right\}. \end{split}\]
The linkage classes are $\mathcal{L}_1 = \left\{ XD,X,XT,X_p \right\}$, $\mathcal{L}_2 = \left\{ X_p + Y, X_pY,\right.$ $\left.X+Y_p \right\}$, and $\mathcal{L}_3 = \left\{ XD+Y_p,XDY_p,XD+Y \right\}$. These can be further decomposed into the the strong linkage classes $\Lambda_1 = \left\{ XD,X,XT \right\}$, $\Lambda_2 = \left\{ X_p \right\}$, $\Lambda_3 = \left\{ X_p+Y,X_pY \right\}$, $\Lambda_4 = \left\{ X+Y_p \right\}$, $\Lambda_5 = \left\{ XD+Y_p,XDY_p\right\}$, and $\Lambda_6 = \left\{ XD + Y \right\}$. Of the strong linkage classes, $\Lambda_2,$ $\Lambda_4$, and $\Lambda_6$ are terminal. It follows that $XD,X,XT,X_p+Y,X_pY,XD+Y_p,XD+Y_p,$ and $XDY_p$, are non-terminal complexes. We can see that the non-terminal complexes $XD$ and $XD + Y_p$ differ only in the species $Y_p$. The network (\ref{system2}) has ten complexes ($n=10$), three linkage classes ($\ell=3$), and the span of the reaction vectors is six-dimensional ($s=6$). It follows that (\ref{system1}) is a deficiency one network.

\subsection{Deterministic chemical reaction systems}

It is common to model the time-evolution of chemical reaction networks via a set of deterministic differential equations over continuous state variables representing the chemical concentrations of the species (i.e. $\textbf{c}_j = [X_j], j=1, \ldots, m$). This modeling choice is suitable when the reaction vessel is well-mixed and the number of all reacting molecules is large \cite{Kurtz2}.

A common kinetic rate assumption is that of the {\it law of mass-action} which states that the rate of a reaction is proportional to the product of the necessary reacting species, counted for multiplicity \cite{G-W}.   Specifically, given the {\it concentration vector} $\mathbf{c} = (c_1,c_2,\ldots,c_m) \in \mathbb{R}_{\geq 0}^m$, the rate of the reaction $y_i \to y_i'$ is
\[
	k_i \mathbf{c}^{y_i} = k_ic_1^{y_{i1}} c_2^{y_{i2}} \cdots c_m^{y_{im}},
\]
where $k_i$ is the rate constant and $y_i$ is the source complex associated with the $i$th reaction channel. 
 Other kinetic rates for the reactions have been used in the biochemical literature, including Michaelis-Menten and Hill kinetics \cite{M-M,Hill}.


Given the chemical reaction network $(\mathcal{S},\mathcal{C},\mathcal{R})$, the associated {\it mass-action system} $(\mathcal{S},\mathcal{C},\mathcal{R},k)$ is given by the system of differential equations
\begin{equation}
\label{de}
\frac{d\mathbf{c}}{dt} = \sum_{i=1}^r k_{i} ( y_i' - y_i ) \mathbf{c}^{y_i}.
\end{equation}
For any $\mathbf{c}_0 \in \mathbb{R}_{>0}^m$, it is known that $\mathbf{c}(t) \in (\mathbf{c}_0 + S) \cap \mathbb{R}_{> 0}^m$ for all $t \geq 0$, where we remind the reader that $S$ is the stoichiometric subspace of the system. Trajectories of (\ref{de}) evolve in a space of dimension $s$, which may be smaller than the state space $\mathbb{R}_{\geq 0}^m$. No trajectory of the deterministic mass-action system (\ref{de}) allows the concentration of any species to hit zero in finite time \cite{V-H}.
 
\vspace{0.2in}

\noindent {\it Example 1:} The mass-action system (\ref{de}) corresponding to the network (\ref{system1}) is
\begin{equation}
\label{system1de}
\begin{split}
\dot {\left(\begin{array}{c}
c_A\\
c_B
\end{array}\right)}
 & = \alpha \left(\begin{array}{c}
 -1\\
 1
\end{array} \right)c_A c_B  + \beta\left(\begin{array}{c}
 1\\
 -1
\end{array} \right)c_B.
\end{split}
\end{equation}

\vspace{0.2in}

\noindent {\it Example 2:} The mass-action system \eqref{de} corresponding to the network (\ref{system2}) is

\begin{equation}
\label{system2de}
\begin{split}
\dot{c}_X & = k_1 c_{XD}-(k_2[D]+k_3[T])c_{X} + k_4 c_{XT}+k_8 c_{X_pY}\\
\dot{c}_{XD} & = - k_1 c_{XD}+k_2[D] c_{X}-k_9 c_{XD}c_{Y_p} +(k_{10}+k_{11})c_{XDY_p}\\
\dot{c}_{XT} & = k_{3}[T] c_X - (k_{4}+k_{5}) c_{XT}\\
\dot{c}_{X_p} & = k_5 c_{XT} - k_6 c_{X_p}c_{Y} + k_7 c_{X_pY}\\
\dot{c}_{Y} & =  - k_6 c_{X_p} c_{Y}+k_7 c_{X_pY} + k_{11} c_{XDY_p}\\
\dot{c}_{X_pY} & = k_{6} c_{X_p}c_Y - (k_{7} + k_{8}) c_{X_pY}\\
\dot{c}_{Y_p} & = k_{8} c_{X_pY} - k_{9} c_{XD}c_{Y_p} + k_{10} c_{XDY_p}\\
\dot{c}_{XDY_p} & = k_{9} c_{XD}c_{Y_p} - (k_{10}+k_{11}) c_{XDY_p},
\end{split}
\end{equation}
where we have dropped  the vector notation.

\subsubsection{Absolute concentration robustness}


A concentration $\bar{\mathbf{c}} \in \mathbb{R}_{\ge 0}^m$ of a mass-action system (\ref{de}) is said to be an {\it  equilibrium concentration} of (\ref{de}) if
\begin{equation}
\label{equilibrium}
\sum_{i=1}^r k_{i} ( y_i' - y_i ) \bar{\mathbf{c}}^{y_i} = 0,
\end{equation}
and said to be a \textit{positive equilibrium concentration if $\bar{\mathbf{c}} \in \mathbb{R}_{> 0}^m$.}
Motivated by biochemical examples in \cite{B-G,Shinar2007,Shinar2009}, Guy Shinar and Martin Feinberg introduce the following classification of positive equilibrium concentrations in \cite{Sh-F}.

\begin{definition}
A mass-action system $(\mathcal{S},\mathcal{C},\mathcal{R},k)$ is said to possess {\it absolute concentration robustness} (ACR) in the species $X_i \in \mathcal{S}$ if $\bar{\mathbf{c}}_i$ attains the same value in every positive equilibrium concentration $\bar{\mathbf{c}} \in \mathbb{R}_{>0}^m$ of (\ref{de}).
\end{definition}

\noindent It is worth noting that the absolutely robust equilibrium value $\bar{\mathbf{c}}_i$ may depend upon the network's rate constants but {\it not} on the initial conditions. This robustness with respect to overall concentrations is especially meaningful for biochemical networks since, when combined with stability, it predicts 
 that processes requiring tight bounds in the robust species will operate similarly 
in the face of fluctuations in the overall concentrations of the other species.

The following result guarantees absolute concentration robustness in a particular species for deterministically modeled mass-action systems and is the main result of \cite{Sh-F}.  It is also the motivation for the current work.
\begin{theorem}
\label{thm:marty}
Consider a mass-action system $(\mathcal{S},\mathcal{C},\mathcal{R},k)$ governed by (\ref{de}). Suppose that:
\begin{enumerate}
\item
The system (\ref{de}) admits a positive equilibrium concentration;
\item
The deficiency of the underlying network $(\mathcal{S},\mathcal{C},\mathcal{R})$ is one (i.e. $\delta=1$); and
\item
There are non-terminal complexes which differ only in the species $X_i$.
\end{enumerate}
Then the mass-action system (\ref{de}) exhibits absolute concentration robustness in $X_i$.
\end{theorem}

\noindent Theorem \ref{thm:marty} is surprising in that it presents {\it structural} conditions on the underlying reaction network which are sufficient to guarantee ACR in a particular species. That is to say, it does not depend on ``fine-tuning''  the network parameters \cite{S-W-S-K}. More details on the theorem, and further examples of ACR networks which do not satisfy Theorem \ref{thm:marty}, are contained in the Supplemental Material of \cite{Sh-F}.

\vspace{0.2in}

\noindent {\it Example 1:} We can easily compute that the only positive equilibrium concentrations permitted by (\ref{system1de}) are
\begin{equation}
\label{equil1}
\begin{split}
\bar{c}_A & = \frac{\beta}{\alpha} \\
\bar{c}_B & = M - \frac{\beta}{\alpha}
\end{split}
\end{equation}
where $M:=c_A(0)+c_B(0)$. It is readily observed that $\bar{c}_A$ has the same value for any positive equilibrium (provided $M > \beta / \alpha$). The system therefore is ACR in the species $A$.
We could also have used Theorem \ref{thm:marty} to find that $A$ is ACR since this network has positive equilibrium concentrations, is deficiency one, and  the non-terminal complexes $A+B$ and $B$ differ only in the species $A$. 

\vspace{0.2in}

\noindent {\it Example 2:} It is shown in the supplemental material of \cite{Sh-F} that every equilibrium concentration $\bar{\mathbf{c}} \in \mathbb{R}_{> 0}^m$ of (\ref{system2de}) satisfies
\begin{equation}
\label{cyp}
\bar{c}_{Y_p} = \frac{k_1k_3k_5(k_{10}+k_{11})[T]}{k_2(k_4+k_5)k_9k_{11}[D]}.
\end{equation}
This concentration depends on the network parameters but not on the concentrations of the other species, which are related by the conservation relationships
\begin{equation}
\label{921}
\begin{split} X_{tot} & = c_X + c_{X_p}+c_{XD}+c_{X_pY}+c_{XDY_p} \\
Y_{tot} & = c_Y + c_{Y_p}+c_{X_pY} + c_{XDY_p}.\end{split}
\end{equation}
It follows that (\ref{system2de}) is ACR in the species $Y_p$ (provided $Y_{tot} > \bar{c}_{Y_p}$ where $\bar{c}_{Y_p}$ is given by (\ref{cyp})).
Again, we could also have used Theorem \ref{thm:marty} to find that $Y_p$ exhibits ACR as (\ref{system2de}) has positive equilibria, the underlying network is deficiency one, and the non-terminal complexes $XD$ and $XD + Y_p$ differ only in the species $Y_p$.

\subsubsection{Reformulation of equation \eqref{de}}

A key feature of CRNT has been in isolating the dependence in the mass-action system (\ref{de}) on the individual complexes $y \in \mathcal{C}$ \cite{H-J1}. It is often convenient therefore to work in ``complex space'' rather than species space. This allows us to reformulate  equation \eqref{de} in a useful manner.  

In order to avoid excessive enumeration, we will index elements in the complex space $\mathbb{R}^{\mathcal{C}}$ explicitly by their corresponding element $y \in \mathcal{C}$. That is to say, for $v \in \mathbb{R}^{\mathcal{C}}$ we will use $v_y$ to denote the element of $v$ corresponding to the complex $y \in \mathcal{C}$. This will allow us to maintain the indexing of complexes introduced previously, i.e. we will still allow $y_i$ and $y_i'$ to correspond to the reactant and product complex, respectively, of the $i$th reaction. We note that $|\mathcal{C}| = n$ so that each vector $v \in \mathbb{R}^{\mathcal{C}}$ has $n$ elements and that, by the {\it support} of an element $v \in \mathbb{R}^{\mathcal{C}}$, we mean the set of those elements of $\mathcal{C}$ for which $v$ takes nonzero values, that is,
\begin{equation}
\label{eq:support_complex}
	\text{supp}(v) := \{y \in \C : v_y \ne 0\}.
\end{equation}
We furthermore denote  the basis vectors $\omega_y$, $y \in \mathcal{C}$, which are the unit vectors with supp$(\omega_y) = y$. The enumeration described above has a long history in chemical reaction network theory \cite{F3}.


Given a choice of vectors $\kappa \in \mathbb{R}_{> 0}^{r}$ and $\Psi \in \mathbb{R}_{\geq 0}^{\mathcal C}$, the {\it kinetic} or {\it Kirchhoff mapping} $A_\kappa: \mathbb{R}^{\mathcal C} \to \mathbb{R}^{\mathcal C}$ of a chemical reaction network $(\mathcal{S},\mathcal{C},\mathcal{R})$ is the linear mapping defined as
\begin{equation}
\label{Ak}
A_{\kappa}(\Psi) = \sum_{y \in \mathcal{C}} \left[ \mathop{\sum_{i=1}^r}_{y = y_i} \kappa_i (\omega_{y_i'}-\omega_{y_i}) \right] \Psi_{y},
\end{equation}
where the inner sum is over the reactions with source complex $y$.
\noindent The formulation (\ref{Ak}) divides the structure of a chemical reaction network according to which complexes act as source complexes for which reactions. If a complex is a source for multiple reactions, those reactions are grouped together. We furthermore define the linear mapping $Y: \mathbb{R}^{\mathcal C} \to \mathbb{R}^m$ by its action on the basis elements $\omega_y$ via
\begin{equation}
\label{Y}
Y (\omega_y) = y.
\end{equation}
It is easy to see that  the composite mapping
\begin{equation}
\label{YA}
Y A_\kappa (\Psi) = \sum_{i=1}^r \kappa_i (y_i'-y_i) \Psi_{y_i}
\end{equation}
is the right-hand-side of the mass-action form (\ref{de}) with $\kappa_i = k_i$ and $\Psi_{y_i} = \mathbf{c}^{y_i}$ for $i=1,\ldots, r$. 

The following two results regarding ker($A_\kappa$) and ker($Y A_\kappa$) can be found in \cite{Sh-F}. The first result is derived in the Appendix of \cite{H-F} and has been restated numerous times since \cite{F3,G-H}. An explicit computation of the basis of ker($A_\kappa$) appears in \cite{C-D-S-S}. The second result follows from Section 6 of \cite{H-F}.

\begin{lemma}
\label{thm:kernel}
Let $(\mathcal{S}, \mathcal{C},\mathcal{R})$ denote a chemical reaction network with terminal strong linkage classes $\Lambda_1, \ldots, \Lambda_t$. Then, for any $\kappa \in \mathbb{R}_{>0}^r$, ker$(A_\kappa)$ has a basis $\left\{ \mathbf{b}_1, \ldots, \mathbf{b}_t \right\}$ where {supp}($\mathbf{b}_\theta) =\Lambda_\theta$ for all $\theta = 1, \ldots, t$.
\end{lemma}

\begin{lemma}
\label{thm:kernel2}
Let $(\mathcal{S},\mathcal{C},\mathcal{R})$ denote a chemical reaction network with deficiency $\delta$ and $t$ terminal strong linkage classes. Then
\begin{equation}
\label{97}
\mbox{dim}(\mbox{ker} (Y A_\kappa)) \leq \delta + t.
\end{equation}
\end{lemma}

\noindent Note that the condition supp$(\mathbf{b}_{\theta}) = \Lambda_{\theta}$ is a set equality property which is well-defined since both are subsets of $\mathcal{C}$ (supp$(\mathbf{b}_{\theta})$ by (\ref{eq:support_complex}) and $\Lambda_{\theta}$ by the definitions of Section \ref{crnsection}).



\subsection{Stochastic models of chemical reaction systems}


Chemical reaction systems can also be modeled with stochastic dynamics as {\it continuous-time Markov chains} (CTMC).  The formulation here closely follows that in \cite{AndKurtz2011}, to which we point the interested reader. In this setting, we let $X_j(t) \in \mathbb{Z}_{\geq 0}$ denote the {\it number} of molecules of $X_j$ at time $t$ and consider the {\it Markov chain} $\mathbf{X}(t) = (X_1(t), X_2(t), \ldots, X_m(t)) \in \mathbb{Z}_{\geq 0}^m$ evolving continuously in time over the discrete state space $\mathbb{Z}_{\geq 0}^m$. At the occurrence of  a reaction $y_i \to y_i'$ the chain instantaneously updates to the new state
\[
	\mathbf{X}(t) = \mathbf{X}(t-) + (y_i'-y_i),
\]
 where $\mathbf{X}(t-)= \lim_{h \to 0^+} \mathbf{X}(t - h)$ is the value of the chain right before the jump. The state of the chain at time $t$ can therefore be given by
\begin{equation}
\label{def:1}
\mathbf{X}(t) = \mathbf{X}(0) + \sum_{i=1}^r N_i(t)(y_i' - y_i)
\end{equation}
where the counting process $N_i(t)$ keeps track of the number of times the $i$th reaction has occurred by time $t$. The counting processes can be further formulated as a function of the state-dependent reaction {\it propensities} $\lambda_i(\mathbf{X}(t))$ by
\begin{equation}
\label{def:2}
N_i(t) = Y_i \left( \int_0^t \lambda_i(\mathbf{X}(s)) \; ds \right),
\end{equation}
where $Y_i(\cdot), i=1, \ldots, r$, are independent, unit-rate Poisson processes.  The propensity functions $\lambda_i$ are the analog of the deterministic rate functions $k_i \mathbf{c}^{y_i}$ in equation \eqref{de},  and are often termed {\it  intensity} functions in the mathematics literature.  Sample trajectories of (\ref{def:1}) and (\ref{def:2}) can be computed by a simulation process known as the next reaction method \cite{G-B,Anderson}, with a similar representation being simulated via the well-known {\it Gillespie algorithm} \cite{Gillespie}.  

As with their deterministic continuous-state counterparts, reactions may only occur in the stochastic setting given sufficient multiplicity of their constituent reactant molecules. This is captured by the following definition.
\begin{definition}
\label{stoichiometricallyadmissible}
The propensities $\lambda_i(\mathbf{X})$, $i=1,\ldots, r$, are said to be {\it stoichiometrically admissible} if $\lambda_i(\mathbf{X}) > 0$ for all $\mathbf{X} \in \mathbb{Z}_{\geq 0}^m$ such that $X_{j} \geq y_{ij}$ for all $j = 1, \ldots, m$, and $\lambda_i(\mathbf{X}) = 0$ otherwise.
\end{definition}
\noindent Although stronger restrictions on the propensities are often made (e.g. monotonicity in the species counts), this definition will be sufficient for the results in this paper. The Markov chain formed by considering a chemical reaction network $(\mathcal{S},\mathcal{C},\mathcal{R})$ together with stoichiometrically admissible propensities is called a {\it stochastic chemical reaction system}. We say a complex is {\it turned off} at a particular state if the propensity of each reaction with that complex as its source is zero at that state value; otherwise we say the complex is {\it turned on}.

A common choice for the propensities $\lambda_i(\mathbf{X})$ is {\it stochastic mass-action}
\begin{equation}
\label{def:massaction}
\lambda_i(\mathbf{X}) = \left\{ \begin{array}{ll} \displaystyle{\frac{k_i}{V^{|y_i| - 1}} \prod_{j=1}^m \frac{X_j!}{(X_j-y_{ij})!y_{ij}!}} \; \; \; & \mbox{if } X_j-y_{ij} \geq 0 \mbox{ for all } j \\ 0, & \mbox{otherwise} \end{array} \right.
\end{equation}
where $|y_i| = \sum_{j=1}^m y_{ij}$ and $V$ is the volume of the reaction vessel. The stochastic chemical reaction system with mass-action propensities (\ref{def:massaction}) is called a {\it stochastic mass-action system}. 

Parametrizing the model \eqref{def:1} by $V$ and denoting by $\mathbf{X}^V(t)$ the model with volume $V$, trajectories $\mathbf{X}^V(t)/V$ of the stochastic model with mass-action propensities (\ref{def:massaction}) are known to converge almost surely on compact time intervals to their deterministic mass-action counterparts \eqref{de} with rate constants $k_i$ in the large-scale limit as $V \to \infty$, so long as $\mathbf{X}_j^V(0)/V$ converges to a non-zero constant for each $j$  \cite{Kurtz3}.

An alternative approach to analyzing stochastic chemical reaction systems is to consider the propagation of the probability of a given chain $\mathbf{X}(t)$ being in a given state as a function of time. We let $P_{t}(\X)=P_t \left\{ \mathbf{X}(t)=\mathbf{X} \right\}$ denote the probability of the chain $\mathbf{X}(t)$ being in the state $\mathbf{X}$ at time $t$. If we define $\lambda_0(\mathbf{X}) = \sum_{i=1}^r \lambda_i(\mathbf{X})$ then, for each state $\mathbf{X} \in \mathbb{Z}_{\geq 0}^m$, we have
\begin{equation}
\label{def:prob}
\frac{dP_{t}(\X)}{dt} = - \lambda_0(\mathbf{X})P_{t}(\X) + \sum_{i=1}^r \lambda_i(\mathbf{X}-(y_i'-y_i))P_{t}(\X-(y_i' - y_i)),
\end{equation}
with an initial condition determined by an initial distribution.
The equation \eqref{def:prob} is called the {\it chemical master equation} in the biology literature, and is called {\it Kolmogorov's forward equations} in the mathematics literature.  Note that the system \eqref{def:prob} can be formally written as
\begin{equation}
\label{def:cme}
\frac{d}{dt}P_t = P_t \A,
\end{equation}
where $P_t$ is the row vector of probabilities of being in a particular state at time $t$ and $\A$ is the generator matrix of the CTMC, which is defined via the above equations. Although (\ref{def:cme}) is a linear system of ordinary differential equations, the scale is typically large (potentially infinite-dimensional) and consequently exact probabilistic solutions to (\ref{def:cme}) are often difficult or impossible to determine explicitly.  Fixed points of equation \eqref{def:cme} are of particular interest as they correspond to stationary distributions of the system and, as a consequence, describe the 
long-term behavior of the chain.

\subsection{Irreducibility, positive recurrence, and stationary distributions}
\label{irreducibilitysection}

The connectivity of states is a key determinant of the long-term behavior of chains $\mathbf{X}(t)$. To that end, we now introduce the following  terminology relevant to how the states of a stochastic chemical reaction system are connected.  
\begin{enumerate}
\item
A state $\mathbf{X} \in \mathbb{Z}_{\geq 0}^m$ {\it has mass} on a complex $y_i \in \mathcal{C}$ if $X_j \geq y_{ij}$ for all $j =1, \ldots, m$;
\item
A state $\mathbf{Y} \in \mathbb{Z}_{\geq 0}^m$ is {\it directly accessible} from $\mathbf{X} \in \mathbb{Z}_{\geq 0}^m$ (denoted $\mathbf{X} \to \mathbf{Y}$) if $\mathbf{Y} = \mathbf{X} + (y_i'-y_i)$ for some $i=1,\ldots, r$ and $\mathbf{X}$ has mass on $y_i$;
\item
A state $\mathbf{Y} \in \mathbb{Z}_{\geq 0}^m$ is {\it accessible} from $\mathbf{X} \in \mathbb{Z}_{\geq 0}^m$ (denoted $\mathbf{X} \leadsto \mathbf{Y}$) if there exists a sequence of states such that $\mathbf{X} = \mathbf{X}_{\mu(1)} \to \mathbf{X}_{\mu(2)} \to \cdots \to \mathbf{X}_{\mu(l)} = \mathbf{Y}$;
\item
The states $\mathbf{X}, \mathbf{Y} \in \mathbb{Z}_{\geq 0}^m$ are said to {\it communicate} (denoted $\mathbf{X} \leftrightsquigarrow \mathbf{Y}$) if $(i)$ $\mathbf{Y} = \mathbf{X}$, or $(ii)$ $\mathbf{X} \leadsto \mathbf{Y}$ and $\mathbf{Y} \leadsto \mathbf{X}$; and
\item
A state $\mathbf{X} \in \mathbb{Z}_{\geq 0}^m$ is {\it recurrent} (respectively, {\it transient}) if the Markov chain starting at $\mathbf{X}(0)=\mathbf{X}$ satisfies
\[P \left\{ T_\mathbf{X} < \infty \; | \; \mathbf{X}(0) = \mathbf{X} \right\} = 1 \; \; \; (< 1)\]
where $T_{\mathbf{X}}$ is the first return time to $\mathbf{X}$. A recurrent state $\mathbf{X} \in \mathbb{R}_{\geq 0}^m$ is furthermore called {\it positive recurrent} (respectively, {\it null recurrent}) if
\[\mathbbm{E} ( T_\mathbf{X} \; | \; \mathbf{X}(0) = \mathbf{X} ) < \infty \; \; \; (= \infty).\]
\end{enumerate}
\noindent That is to say, a state $\mathbf{X} \in \mathbb{Z}_{\geq 0}^m$ is recurrent if the chain $\mathbf{X}(t)$ satisfying $\mathbf{X}(0) = \mathbf{X}$ is guaranteed to return to $\mathbf{X}$, and positive recurrent if the expected return time is finite.


The relation ``$\leftrightsquigarrow$'' allows the state space $\mathbb{Z}_{\geq 0}^m$ to be partitioned into {\it irreducible communicating classes} or {\it irreducible components}. The states $\mathbf{X}, \mathbf{Y} \in \mathbb{Z}_{\geq 0}^m$ are in the same irreducible component $\mathcal{I} \subseteq \mathbb{Z}_{\geq 0}^m$ if and only if $\mathbf{X} \leftrightsquigarrow \mathbf{Y}$. An irreducible component $\mathcal{I} \subseteq \mathbb{Z}_{\geq 0}^m$ is said to be {\it closed} if $\mathbf{X} \not\leadsto \mathbf{Y}$ for all $\mathbf{X} \in \mathcal{I}$ and $\mathbf{Y} \not\in \mathcal{I}$, and is said to be {\it absorbing} if it is closed and $\mathbf{Y} \to \mathbf{X}$ for some $\mathbf{X} \in \mathcal{I}$ and $\mathbf{Y} \not\in \mathcal{I}$. It is a standard result from introductory texts on continuous time Markov chains that for any irreducible component $\mathcal{I} \subseteq \mathbb{Z}_{\geq 0}^m$, if $\mathbf{X} \leadsto \mathbf{Y}$ for $\mathbf{X},\mathbf{Y} \in \mathcal{I}$ and $\mathbf{X}$ is (positive) recurrent, then $\mathbf{Y}$ is (positive) recurrent, and that if $\mathbf{Y}$ is transient then $\mathbf{X}$ is transient \cite{Lawler}. Consequently, recurrence, positive recurrence, and transience are class properties of irreducible components \cite{Lawler}.



In order to determine the long-term behavior of stochastic processes it is important to consider the {\it stationary distributions} $\pi = (\pi(\mathbf{X}))_{\mathbf{X} \in \mathbb{Z}_{\geq 0}^m}$, which are the non-negative fixed points of (\ref{def:cme}), normalized to sum to one.  Specifically, $\pi$ is a stationary distribution if it satisfies the following 
\begin{enumerate}[(i)]
\item
$\sum_{\mathbf{X} \in \mathbb{Z}_{\geq 0}^m} \pi(\mathbf{X}) = 1$;
\item
$0 \leq \pi(\mathbf{X}) \leq 1$ for all $\mathbf{X} \in \mathbb{Z}_{\geq 0}^m$; and
\item $\pi \mathcal A = 0$ where $\mathcal A$ is the generator of the CTMC (\ref{def:cme}).
\end{enumerate}
It is a standard result that the stationary distributions of continuous-time Markov chains, if they exist, are restricted to the support of the closed irreducible components and are unique on these components. That is to say, if $\mathcal{I}_1, \ldots, \mathcal{I}_p$, where $p$ could be infinity, are closed irreducible components of a continuous time Markov chain $\mathbf{X}(t)$, then each component has a unique stationary distribution $\pi_{\mathcal{I}_j} = (\pi(\mathbf{X}))_{\mathbf{X} \in \mathcal{I}_j}$, $j=1, \ldots, p$, and any stationary distribution $\pi$ of the process can be expressed as
\begin{equation}
\label{def:stationary2}
\pi = \sum_{j=1}^p \gamma_j \pi_{\mathcal{I}_j},
\end{equation}
where $0 \leq \gamma_j \leq 1$ and $\sum_{j=1}^p \gamma_j = 1$.  See, for example,  \cite{Lawler}.

\section{Main results}
\label{sec:Main_Results}

In this section we state and prove our main results, Theorem \ref{thm:main_general} and Theorem \ref{thm:cor1}.

We begin in Section \ref{robustnesssection} by motivating the notion that deterministic and stochastic models can exhibit fundamentally different limiting behavior through consideration of the networks (\ref{system1}) and (\ref{system2}).  Next, in Section \ref{proofsection}, we provide statements and proofs of our main results, which provide conditions for when an absorption event is guaranteed to occur for a stochastically modeled network.  In Section \ref{sec:after} we provide a theorem detailing the deficiency, connectivity properties, and long-term behavior of the \textit{post-absorption} network.  In Section \ref{sec:notes} we show three things: (i) that Theorem \ref{thm:main_general}  may be successfully applied to networks more general than those covered by Theorem \ref{thm:cor1}, which is stated in the main article text, (ii)  that the results of Theorem \ref{thm:main_general} and Theorem \ref{thm:cor1} may fail for networks for which a global conservation relationship does not hold, and (iii) that the conclusions of Theorems \ref{thm:main_general} and \ref{thm:cor1} hold for more ACR models than those which fit the assumptions of the theorems, leading to a conjecture.  In Section \ref{analysissection}, we consider the time until the guaranteed absorption takes place, and the quasi-stationary distributions of the network models.


\subsection{Differences in long-term behavior for ACR systems}
\label{robustnesssection}

We are interested in when the long-term behavior of deterministically modeled networks, as characterized by their stable equilibria, is qualitatively different than the long-term behavior of stochastically modeled networks, as characterized by their stationary distributions and recurrence properties.  Specifically, we are interested in whether or not trajectories of the stochastic  systems become concentrated around deterministically predicted stable equilibrium concentrations.

\vspace{0.2in}

\noindent {\it Example 1:} Trajectories of the stochastic chemical reaction system corresponding to (\ref{system1}) have different long-term behavior than  those of the deterministically modeled mass-action system (\ref{system1de}). For all initial counts $X_A(0)$ and $X_B(0)$, we can see that the state
\begin{equation}
\label{system1boundary}
(\bar{X}_A,\bar{X}_B) = (M,0),
\end{equation}
where $M:= X_A(0) + X_B(0)$, is a trapping state since neither reaction may proceed from this state. This corresponds to an inaccessible boundary equilibrium concentration in the deterministic setting \cite{V-H}.

The state (\ref{system1boundary}) is also, however, an {\it accessible} state for the corresponding stochastic mass-action system. From any positive state $(X_A(t), X_B(t))$ of the chain  it is possible to transition irreversibly through the repeated occurrences of $B \to A$ to the state (\ref{system1boundary}). It follows that, for any initial counts, (\ref{system1boundary}) is the unique absorbing state of the chain. Hence, for any value $M>0$, the unique stationary distribution only has mass on \eqref{system1boundary} and not near the  positive equilibrium $(\bar{c}_A,\bar{c}_B) = (\beta/\alpha,M-\beta/\alpha)$ of the deterministic model.

\vspace{0.2in}

\noindent {\it Example 2:} Although more challenging to see, trajectories of the stochastic chemical reaction system corresponding to the system (\ref{system2}) also have different long-term behavior than those of the deterministic model (\ref{system2de}). Consider the state
\begin{equation}
\label{system2boundary}
\begin{split}
X_p & = X_{tot}>0,\\
Y_p & = Y_{tot}>0,\\
X & =Y=XD=XT=X_pY=XDY_p=0
\end{split}
\end{equation}
where $X_{tot}$ and $Y_{tot}$ are given by
\begin{equation}
\label{922}
\begin{split} X_{tot} & = X(0) + X_p(0)+XD(0)+X_pY(0)+XDY_p(0) \\
Y_{tot} & = Y(0) + Y_p(0)+X_pY(0) + XDY_p(0).\end{split}
\end{equation}
It is clear that this is an absorbing state since all reaction propensities are zero at (\ref{system2boundary}). Modeled deterministically, the corresponding state is a boundary equilibrium that is not accessible from the positive orthant \cite{V-H}. In the stochastic setting, however, the state can be shown to be accessible in the following manner. Start with arbitrary positive counts for every species and suppose the following reactions take place sequentially:
\begin{enumerate}
\item
Convert all $XDY_p$ into $Y$ and $XD$ through the reaction $XDY_p \to XD+ Y$.
\item
Convert all $XD, X,$ and $XT$ to $X_p$ through the chain of reactions $XD \to X \to XT \to X_p$.
\item
Convert as much $Y$ into $Y_p$ as possible through $X_p + Y \to X_p Y \to X + Y_p$.
\item
Repeat $2.$ and $3.$ until all $Y$ is converted to $Y_p$ and all $XD$, $X$, and $XT$ is converted to $X_p$.
\end{enumerate}
This algorithm converts all derivatives of $X$ into $X_p$ and all derivatives of $Y$ into $Y_p$. Although such a chain of reactions is a low probability event, it is irreversible and
by standard probabilistic arguments the process will converge to the absorbing state with a probability of one.   Note also that this conclusion holds for any choice of stoichiometric admissible propensities, not just mass-action propensities (\ref{def:massaction}). 


\vspace{0.2in}

These examples demonstrate that the stability of robust equilibrium concentrations predicted for deterministic mass-action systems (\ref{de}) may not be taken for granted when the underlying stochastic nature of the chemical reactions should not be ignored. In the next section we show that this difference in long-term behavior is the rule, as opposed to the exception.  In Section \ref{analysissection} we discuss both the time until such absorption takes place and, in the case of a large expected time until absorption,  the  methods used to understand the behavior of these processes before absorption takes place.

\subsection{Statements and proofs of main results}
\label{proofsection}

We restate two definitions presented in the main article text.

\begin{definition}
\label{domination}
	We say that the complex $y \in \mathcal{C}$ is {\it dominated} by the complex $y' \in \mathcal{C}$ (denoted $y \ll y'$) if $y_i' \leq y_i$ for all $i=1, \ldots, m$.
\end{definition}

\begin{definition}
\label{conservative}
A chemical reaction network $(\mathcal{S},\mathcal{C},\mathcal{R})$ is said to be {\it conservative} if there exists a vector $w \in \mathbb{R}_{>0}^m$ such that $w \cdot (y'-y) = 0$ for all $y\to y' \in \Re$.
\end{definition}

\noindent We again note that  whenever two complexes differ in only one species, one necessarily dominates the other.  
For example, the complex $X_1+X_2$ is dominated by the complex $X_1$ because we have that $y = (1,1)$ and $y' = (1,0)$ satisfy $y_i' \leq y_i$ component-wise.  Note that if $y \ll y'$ and there are insufficient molecular counts for reactions out of $y'$ to proceed, then there are necessarily insufficient counts for reactions out of $y$ to proceed. We also reiterate the intuition that Definition \ref{conservative} implies the existence of a conserved quantity $M>0$ so that for the given species set $\left\{ X_1, X_2, \ldots, X_m \right\}$ we have that
\[M:=w_{1} c_{1}(t)+w_{2} c_{2}(t) + \cdots + w_{m} c_{m}(t)\]
is invariant to the dynamics of the system.

The following is our main technical result.  Note that it reduces the problem of determining the long-term dynamics of a stochastically modeled system to one of linear algebra.

\begin{theorem}
\label{thm:main_general}
Let $(\mathcal{S},\mathcal{C},\mathcal{R})$ be a conservative chemical reaction network for which the following assumptions hold:

\begin{enumerate}

\item There are non-empty sets of non-terminal complexes $\mathcal{C}^*$ and $\mathcal{C}^{**}$ such that, if $y \in \mathcal{C}^*$ then there exists a $y' \in \mathcal{C}^{**}$ such that $y \ll y'$.



\item  For some choice of rate constants $\left\{ k_i \right\}_{i=1}^r$, the following property holds: if $\Psi \in \mbox{ker}(Y A_k)$ and $\Psi_{y}=0$ for all $y \in \mathcal{C}^*$, then $\Psi_{\bar{y}}=0$ for all non-terminal $\bar{y}$.


\end{enumerate}
Then, for any choice of stoichiometrically admissible kinetics, 
	all non-terminal complexes of the network are off at each positive recurrent state of the stochastically modeled system.
\end{theorem}

\begin{remark}
	The second condition in assumption 2 above may also be formulated in the following way, 
which may be more intuitive to some readers.
	\begin{enumerate}
	\item[$2'$.] For some choice of rate constants $\left\{ k_i \right\}_{i=1}^r$, the following property holds: if $\Psi \in \mbox{ker}(Y A_k)$  has support on a non-terminal complex, then $\Psi$  has support on some $y \in \mathcal{C}^*$.
	\end{enumerate}
\end{remark}


 We know that when complexes are off 
we have that at least one chemical species is ``near'' zero.  
This result therefore says that trajectories of the stochastic networks satisfying the required conditions will, with a probability of one, get stuck near the boundary of the positive orthant, even if the deterministic model predicts stability of a robust positive equilibrium concentration. This results therefore represents a fundamental disparity in the long-term predictions of the corresponding deterministically and stochastically modeled systems satisfying the above structural conditions.

 It can be easily checked that the networks (\ref{system1}) and (\ref{system2}) satisfy the assumptions of Theorem \ref{thm:main_general}. To motivate the proof of Theorem \ref{thm:main_general} we first develop some intuition based on these networks. A key observation in showing the boundary states (\ref{system1boundary}) and (\ref{system2boundary}) were accessible was that we never used a reaction from the dominated non-terminal complex (i.e. we never used the reaction $A+B \to 2B$ or $XD + Y_p \to XDY_p$, respectively). Consequently, when constructing our sequence of reactions we were actually considering a subnetwork which excluded these reactions. This suggests the following algorithm for showing the boundary is accessible:
\begin{enumerate}
\item
Identify the sets of non-terminal complexes $\mathcal{C}^*$ and $\mathcal{C}^{**}$ satisfying assumption 1 of Theorem \ref{thm:main_general}.
\item
 Remove all reactions out of the complexes in $\mathcal{C}^*$.  
\item
Use condition {\it 2.}  to show that  all non-terminal complexes of the reduced network are off at each positive recurrent state; in particular, show that no positive recurrent state has mass on any $y' \in \mathcal{C}^{**}$. Furthermore, show that the conservation property guarantees the accessibility of some set of positively recurrent states.
\item
Use the condition $y \ll y'$ to conclude the result for the original system.
\end{enumerate}
 The removal of the reactions will allow us to proceed with a proof by contradiction.  Specifically, by consideration of the \textit{reduced network} we will be able to construct a vector contradicting  condition 2.\ in the statement of the theorem.


The idea of removing reactions from a network is captured by the following definition.
\begin{definition}
Let $\mathcal{D}^* \subset \mathcal{C}$ denote a set complexes of the chemical reaction network $(\mathcal{S},\mathcal{C},\mathcal{R})$ and $\mathcal{R}' \subseteq \mathcal{R}$ denote the set of all reactions in $\mathcal{R}$ for which some complex in $\mathcal{D}^*$ is the source complex. We define the {\it $\mathcal{D}^*$-reduced subnetwork} to be the triple $(\mathcal{S}, \mathcal{C},\mathcal{R}^*)$ where $\mathcal{R}^* = \mathcal{R} \setminus \mathcal{R}'$.
\end{definition}

\noindent In other words, for each $y \in \mathcal{D}^*$, we remove from the network the set of reactions for which $y$ is the reactant complex. Note that we have not changed the set of species or complexes, even though we may now have species and/or complexes not involved in any reaction (and so this is no longer technically a reaction network by Definition \ref{crn}).  

For example, if the original network is
\[
	A \leftarrow 2B \rightarrow C, \ \ \ \ B \to D, 
\]
which has species $\mathcal {S} = \{A,B,C,D\}$, complexes $\mathcal {C} = \{A,2B,C,B,D\}$, and reactions as given above, then the $\left\{ 2B \right\}$-reduced network is
\[
	A, \quad 2B, \quad C, \ \ \ \ B \to D,
\]
with species $\mathcal {S} = \{A,B,C,D\}$, complexes $\mathcal {C} = \{A,2B,C,B,D\}$, and reduced reaction set $\{B \to D\}$.  The $\{2B,B\}$-reduced network is simply
\[
	A, \quad 2B, \quad C, \ \ \ \ B, \quad D.
\]

\begin{proof}[Proof of Theorem \ref{thm:main_general}]

Fix a choice of rate constants  $\{k_i\}_{i=1}^r$   for which the property of assumption 2 in the statement of the theorem holds. 
The proof will proceed by first showing that the conclusion of the theorem holds under the assumption of stochastic mass-action kinetics with these rate constants.  We then extend to any admissible kinetics.

Let $\mathcal{C}^*$ and $\mathcal{C}^{**}$ be the non-empty sets satisfying assumption 1 of Theorem \ref{thm:main_general}.     Let $(\mathcal{S},\mathcal{C},\mathcal{R}^*)$ denote the {\it $\mathcal{C}^*$-reduced subnetwork} of $(\mathcal{S},\mathcal{C},\mathcal{R})$.  Let $\mathbf{X}^*(t)$ denote the continuous time Markov chain associated with the $\mathcal{C}^*$-reduced subnetwork with stochastic mass-action kinetics and rate constants $k_i$ if $y_i \to y_i' \in \mathcal{R}^*$.  

 For the time being we will only consider the chain $\mathbf{X}^*(t)$ for the reduced network.  We will show that there does not exist a positive recurrent state for the chain $\mathbf{X}^*(t)$ which has mass on a non-terminal complex.  This portion of the proof will proceed by contradiction.  That is, we first suppose that there is a positive recurrent $\mathbf{X}_0$ which has mass on a non-terminal complex, and will eventually conclude that assumption 2 is violated.    After showing the result for $\mathbf{X}^*(t)$, we will show the result holds for the original chain with any choice of stoichiometrically admissible kinetics.
 
 Let ${\mathcal C}^*_{\text{source}} \subset \mathcal{C}$ denote the set of complexes which are source complexes for reactions in the $\mathcal{C}^*$-reduced network and let $\mathcal I$ be the irreducible component containing $\mathbf{X}_0$.  Note that, by construction, 
 \begin{align}\label{eq:need_11332}
 	{\mathcal C}^*_{\text{source}} \cap \mathcal C^* = \emptyset. 
\end{align}
 We define 
\[
	\mathcal{C}_\mathcal{I} := \{ y \in \mathcal{C} \; | \;  y \in \mathcal{C}^*_{\text{source}} \text{ and there is an $\mathbf{X} \in \I$ with mass on $y$}\}.
\]
These are the relevant source complexes on the state space $\mathcal I$.
Note that our assumption that $\mathbf{X}_0$ has mass on a non-terminal complex of the reduced network implies that there is at least one non-terminal complex contained in $\mathcal{C}_\I$.


Now consider the process $\mathbf{X}^*(t)$ with initial condition $\mathbf{X}^*(0) = \mathbf{X}_0 \in \mathcal{I}$. We denote the $N$th return time to $\mathbf{X}_0$ as $t_N$. It follows that we have $\mathbf{X}^*(t_N) = \mathbf{X}^*(0)$ for all $N$ so that (\ref{def:1}) gives
\begin{equation}\label{dtn}
	\mathbf{X}^*(t_N) = \mathbf{X}^*(0) + \mathop{\sum_{i=1}^r}_{y_i \in \mathcal{C}_{\mathcal{I}}} N_i(t_N) (y_i' - y_i) = \mathbf{X}^*(0),
\end{equation}
where we are only summing over those reactions that admit positive propensities for at least one state in $\I$ (all other propensities are identically zero for all time).  Since we are for the time being assuming mass-action kinetics (\ref{def:massaction}), we have from (\ref{def:1}), (\ref{def:2}), and (\ref{dtn}) that for each $N>0$,
\begin{equation}
\label{eqn:3}
\mathop{\sum_{i=1}^r}_{y_i \in \mathcal{C}_{\mathcal{I}}} Y_i \left( \frac{k_i}{V^{|y_i| - 1}}  \int_0^{t_N} \prod_{j=1}^m \frac{\mathbf{X}^*_j(s)!}{(\mathbf{X}^*_j(s)-y_{ij})!y_{ij}!} ds \right) (y_i' - y_i) = 0,
\end{equation}
where $Y_i(\cdot)$ are independent unit-rate Poisson processes.

  For any complex $y \in \mathcal{C}$, we define the function $G_{y}: \R^m \to \mathbb{R}_{\ge 0}$, via
\begin{equation}
\label{def:G}
G_{y}(\mathbf{X}) =  \left\{
\begin{array}{cl}
	\displaystyle\frac{1}{V^{|y| - 1}} \prod_{j=1}^m \frac{\mathbf{X}_j!}{(\mathbf{X}_j-{y}_{j})! {y}_{j}!}, \; \; \; & \text{for } y \in \mathcal{C}_{\mathcal I} \text{ and } \mathbf{X} \in \mathcal{I}\\
	0 & \text{otherwise.}
	\end{array}
	\right.
\end{equation}
We note two things.  First, $G_{y} \equiv 0$ whenever $y \in \mathcal{C} \setminus \mathcal{C}_{\mathcal{I}}$. In particular, by \eqref{eq:need_11332} we have that $G_{y} \equiv 0$ for all $y \in \mathcal{C}^*$.  
Second, note that (\ref{def:G}) does not depend on the reaction rate constants; rather, it quantifies the intensity of the state-dependent portion of the mass-action term (\ref{def:massaction}) of all the reactions  with source complex $y$. 

 Returning to (\ref{eqn:3}), by multiplying and dividing by appropriate terms, we find 
{\small
\begin{equation}
\label{eqn:4}
\mathop{\sum_{i=1}^r}_{y_i \in \mathcal{C}_{\mathcal{I}}} k_i \left[ \frac{1}{\{k_i \int_0^{t_N} G_{y_i}(\mathbf{X}^*(s)) ds \} \vee 1} Y_i \left( k_i \int_0^{t_N} G_{y_i}(\mathbf{X}^*(s)) ds \right) \times \{ \frac{1}{t_N} \int_{0}^{t_N} G_{y_i}(\mathbf{X}^*(s)) ds\} \vee 1 \right] (y_i' - y_i) = 0,
\end{equation}
}

\noindent where $a\vee b = \max\{a,b\}$.  Since $\I$ is a positive recurrent class, for any $y_i \in \mathcal{C}_\mathcal{I}$ we have that
\[
	\int_0^{t_N} G_{y_i}(\mathbf{X}^*(s)) ds \to \infty, \text{ almost surely as } N \to \infty,	
\]
and so
\begin{equation}
\label{eqn:5}
\lim_{N \to \infty} \frac{1}{\{ k_i \int_0^{t_N} G_{y_i}(\mathbf{X}^*(s)) ds\} \vee 1} Y_i \left( k_i \int_0^{t_N} G_{y_i}(\mathbf{X}^*(s)) ds \right) = 1, \; \; \; \mbox{almost surely as $N \to \infty$},
\end{equation}
where we have applied the law of large numbers to the Poisson process $Y_i$. Denoting $\pi_{\I}$ as the unique stationary distribution of $\mathbf{X}^*$ on $\mathcal{I}$, by basic ergodicity properties of continuous time Markov chains we have that
\begin{equation}
\label{eqn:6}
\lim_{N \to \infty} \frac{1}{t_N} \int_0^{t_N} G_{y_i}(\mathbf{X}^*(s)) ds = \sum_{\mathbf{X} \in \mathcal{I}} \pi_{\mathcal{I}}(\mathbf{X}) G_{y_i}(\mathbf{X}),
\end{equation}
for all complexes $y_i\subset \mathcal C$.
\noindent We now define  the vector $G_{\pi_{\mathcal I}} \in \R^\C$ via
\begin{equation}
\label{eqn:7}
\left[G_{\pi_{\mathcal{I}}}\right]_{y} :=  \sum_{\mathbf{X} \in \mathcal{I}} \pi_{\mathcal{I}}(\mathbf{X}) G_{y}(\mathbf{X}), 
\end{equation}
for $y \in \mathcal{C}$ and note that $[G_{\pi_{\mathcal{I}}}]_{y} > 0$ if and only if $y \in \mathcal{C}_{\mathcal{I}}$, and we reiterate that by our original assumptions $\mathcal{C}_{\mathcal{I}}$ contains at least one non-terminal complex, but does not contain any $y \in \mathcal C^*$.

 Combining (\ref{eqn:4}) - \eqref{eqn:7}, we have
\begin{align*}
 \mathop{\sum_{i=1}^r}_{y_i \in \mathcal{C}_{\mathcal{I}}} k_i (y_i' - y_i) [G_{\pi_{\I}}]_{y_i}=0,
\end{align*}
and so
\begin{align*}
 0= \sum_{y \in \mathcal{C}_{\mathcal{I}}} \left[ \mathop{\sum_{i=1}^r}_{y = y_i} k_i (y_i'-y_i) \right][G_{\pi_{\mathcal{I}}}]_{y} = \sum_{y \in \mathcal{C}} \left[ \mathop{\sum_{i=1}^r}_{y = y_i} k_i (y_i'-y_i) \right][G_{\pi_{\mathcal{I}}}]_{y},
\end{align*}
where the first equality is a reordering of terms, and the second equality follows since $[G_{\pi_{\mathcal I}}]_{y} =0$ for $y \notin \mathcal{C_I}$.
Hence, the vector $G_{\pi_{\mathcal{I}}} \in \R^{\mathcal C}$ is contained in the kernel of $YA_k$.  By construction, however, we have 
\begin{enumerate}[$(i)$]
\item  $[G_{\pi_{\mathcal{I}}}]_{y} = 0$ for each $y \in \mathcal{C}^*$.
\item  $[G_{\pi_{\mathcal{I}}}]_{\bar y}>0$ for at least one non-terminal complex,  
\end{enumerate}
which contradicts assumption 2 in the statement of the theorem.
  Consequently, we can conclude that there are no positive recurrent states of the $\mathcal{C}^*$-reduced network with mass on any non-terminal complexes, including any  $y' \in \mathcal{C}^{**}$.
Note that it was by considering the reduced network, as opposed to the original network, that allowed us to conclude $(i)$ above, which led us to the contradiction.


\vspace{.1in}
The above analysis shows that the  conclusion of the theorem holds for the $\mathcal{C}^*$-reduced subnetwork $(\mathcal{S},\mathcal{C},\mathcal{R}^*)$ taken with mass-action propensities (\ref{def:massaction}).  We now argue that the result holds for the original network $(\mathcal{S},\mathcal{C},\mathcal{R})$ with any stoichiometric admissible propensities. We make the following observations:
\begin{enumerate}[$(i)$]
\item
Any sequence of transitions which can occur in the $\mathcal{C}^*$-reduced subnetwork $(\mathcal{S},\mathcal{C},\mathcal{R}^*$) can occur in the original network $(\mathcal{S},\mathcal{C},\mathcal{R})$. In other words, for any two states $\mathbf{X},\mathbf{Y} \in \mathbb{Z}_{\geq 0}^m$ such that $\mathbf{X} \leadsto \mathbf{Y}$ for $(\mathcal{S},\mathcal{C},\mathcal{R}^*)$, we have $\mathbf{X} \leadsto \mathbf{Y}$ for $(\mathcal{S},\mathcal{C},\mathcal{R})$. 
\item
By the domination property $y \ll y'$, any state $\mathbf{X} \in \mathbb{Z}_{\geq 0}^m$ which does not have mass on any $y' \in \mathcal{C}^{**}$ also does not have mass on any $y \in \mathcal{C}^*$.
\item
The connectivity properties of a stochastic chemical reaction system with mass-action propensities and  one with general stoichiometrically admissible propensities are identical. That is to say, for two states $\mathbf{X}, \mathbf{Y} \in \mathbb{Z}_{\geq 0}^m$, $\mathbf{X} \to \mathbf{Y}$ for a stochastic mass-action system if and only if $\mathbf{X} \to \mathbf{Y}$ for the corresponding stochastic reaction system with general stoichiometrically admissible propensities.
\end{enumerate}

Now consider any closed, irreducible, positive recurrent component $\mathcal{I}$ of the state space of the stochastic reaction system $(\mathcal{S},\mathcal{C},\mathcal{R}^*)$, where we are for the time being still assuming mass-action kinetics with the specified rate constants. By our previous argument, we have that $\mathcal{I}$ does not contain a state which has mass on $y' \in \mathcal{C}^{**}$. Hence, it follows from $(ii)$ that $\mathcal{I}$ does not contain a state that has mass on any $y \in \mathcal{C}^*$. Furthermore, since $\mathcal{R}$ and $\mathcal{R}^*$ differ only in reactions with source complexes $y \in \mathcal{C}^*$, we have exactly the same connectivity properties on $\mathcal{I}$ for $(\mathcal{S},\mathcal{C},\mathcal{R})$ as we do for the  $\C^*$-reduced subnetwork $(\mathcal{S},\mathcal{C},\mathcal{R}^*)$. It follows that $\mathcal{I}$ is a closed irreducible component of the stochastic reaction system $(\mathcal{S},\mathcal{C},\mathcal{R})$ and, since $\mathcal{I}$ was chosen arbitrarily, it follows that every closed irreducible component of $(\mathcal{S},\mathcal{C},\mathcal{R}^*)$ is a closed irreducible components of $(\mathcal{S},\mathcal{C},\mathcal{R})$. It furthermore follows from $(i)$ that these are the only positive recurrent communication classes for $(\mathcal{S},\mathcal{C},\mathcal{R})$.  Hence, the result is shown for the system $(\S,\C,\Re)$ when the kinetics are mass-action with rate constants $k_i$. We now note that by $(iii)$, the generalization to any choice of stoichiometrically accessible kinetics is trivial.  
\end{proof}

Thus, we conclude that all trajectories will be absorbed by states for which only the terminal complexes can be the sources (and hence the products) of reactions. An immediate corollary to Theorem \ref{thm:main_general} is the following, which is Theorem 2 in the main article text.
\begin{theorem}
\label{thm:cor1}

Let $(\mathcal{S},\mathcal{C},\mathcal{R})$ be a conservative chemical reaction network for which the following assumptions hold:   
\begin{enumerate}
\item 
For some choice of rate constants $\{k_i\}_{i = 1}^r$, the deterministically modeled mass-action system admits a positive equilibrium;
\item
The deficiency of the network is one (i.e. $\delta=1$);

\item There are two non-terminal complexes, $y$ and $y'$ say, for which $y \ll y'$.
\end{enumerate}
Then, for any choice of stoichiometrically admissible kinetics, 
	all non-terminal complexes of the network are off at each positive recurrent state of the stochastically modeled system.
\end{theorem}

\begin{remark}
Since the conditions of Theorem \ref{thm:marty} are a subset of these conditions, this result says that every conservative network guaranteed to exhibit absolute concentration robustness by Theorem \ref{thm:marty} in the deterministic setting has trajectories $\mathbf{X}(t)$ which eventually have an absorption event
 in the stochastic setting. Notably, this is true even when the deterministic model predicts stability of the positive ACR concentration.
\end{remark}

\begin{proof}[Proof of Theorem \ref{thm:cor1}]  
	Let $\{\Lambda_1,\dots, \Lambda_t\}$ be the terminal strong linkage classes of the network.  It is sufficient to  show that for some choice of rate constants $\{k_i\}_{i=1}^r$, the support of each element of the kernel of $YA_{k}$ either contains all the non-terminal complexes, or none of them.   The arguments we employ are  taken from \cite{Sh-F}.

By assumption 1 there is a set of rate constants, $\{k_i\}_{i=1}^r$, for which the deterministically modeled system \eqref{de} admits a positive equilibrium, $\mathbf{\bar c} \in \R^m_{>0}$.  That is,     
\begin{equation}\label{eq:equilibrium_1}
	\sum_{i=1}^r k_i (y_i' - y_i) \bar{\mathbf{c}}^{y_i} = \sum_{y \in \mathcal{C}} \left[ \mathop{\sum_{i=1}^r}_{y = y_i} k_i (y_i'-y_i) \right] \bar{\mathbf{c}}^{y}=0,
\end{equation}
 where the inner sum is over all reactions out of complex $y$. It follows from (\ref{Ak}) that $\bar{\mathbf{c}}^Y \in \mbox{ker}(Y A_k)$ where $\bar{\mathbf{c}}^Y \in \mathbb{R}^{\mathcal C}$ is the vector with entries $[\bar{\mathbf{c}}^Y]_y = \bar{\mathbf{c}}^y$. By assumption $2$ (that the deficiency of the network is one) and Lemma \ref{thm:kernel2} we have that $\mbox{dim}(\mbox{ker} (Y A_k)) \leq 1 + t$. It follows from Lemma \ref{thm:kernel} that $\mbox{ker} (Y A_k)$ has a basis $\left\{ \bar{\mathbf{c}}^Y,\mathbf{b}_1, \ldots, \mathbf{b}_t \right\}$ where the $\mathbf{b}_i$, $i=1, \ldots, t$, have support on the $i$th terminal strong linkage class, $\Lambda_i$. Thus, any $\Psi \in \mbox{ker}(Y A_k)$ can be written
\begin{equation}
\label{eqn:99}
\Psi = \lambda_0 \bar{\mathbf{c}}^Y + \sum_{\theta=1}^t \lambda_\theta \mathbf{b}_\theta,
\end{equation}
where $\lambda_i \in \mathbb{R}$, $i=0,\ldots, t$, are constants. It is clear that $\bar{\mathbf{c}}^Y$ has support on all the non-terminal complexes and the result is shown. 
\end{proof}

\subsection{The post-absorption network}
\label{sec:after}
We now consider the structure and long-term behavior of the post-absorption network.
Let $\{\S,\C,\Re\}$ be a reaction network satisfying the assumptions of Theorem \ref{thm:cor1}.  Let $\C'\subset \C$ denote those complexes that are both (i) contained in 
a terminal strong linkage class and (ii) the source complex for some reaction.  Note that due to the definition of a terminal strong linkage class, each element of $\C'$ is also a product complex for some reaction.  Let $\Re'$ denote the reactions with source complexes $\C'$ and denote by $\S'$ those species in the support of some element of $\C'$.  Then $\{\S',\C',\Re'\}$ is the reaction network which could still be ``on'' after the absorption event guaranteed by Theorem \ref{thm:cor1}.  We note that  $\{\S',\C',\Re'\}$ could be trivial in that $\S' = \C' = \Re' = \emptyset$, otherwise we say it is \textit{non-trivial}.  For example, for both the networks in our running examples \eqref{system1} and \eqref{system2}, the post-absorption network is trivial; whereas, the post-absorption network for 
\begin{equation}\label{eq:post_nontrivial}
	X \rightleftarrows Y, \ \ \ A \to B + X, \ \ \ \ A + B + X \to 2A, 
\end{equation}
which satisfies the assumptions of Theorem \ref{thm:cor1}, is $X \rightleftarrows Y.$  Note that the network \eqref{eq:post_nontrivial} does not satisfy the requirements of Theorem \ref{thm:marty} and can be checked explicitly to not exhibit ACR in any species when modeled deterministically.

The following lemma, combined with results from \cite{A-C-K},  characterizes the limiting behavior of the stochastically modeled system after absorption assuming stochastic mass-action kinetics: a stationary distribution is 
\[
	\pi'(x) = C \prod_{i = 1}^{|\S'|} \frac{(c'_i)^{x_i}}{x_i!}, \ \ \ x \in \Gamma,
\]
subject to the necessary conservation relations, where $\Gamma$ is the resulting state space, $c'$ is a positive equilibrium of the \textit{deterministically} modeled mass-action system associated with the network $\{\S',\C',\Re'\}$, and $C$ is a normalizing constant.

\begin{lemma}
	Suppose the assumptions of Theorem \ref{thm:cor1} hold for the network $\{\S,\C,\Re\}$ and that the post-absorption network  $\{\S',\C',\Re'\}$ is non-trivial.  Then $\{\S',\C',\Re'\}$ is weakly reversible and has a deficiency of zero.
\end{lemma}

\begin{proof}
	The network $\{\S',\C',\Re'\}$ is weakly reversible by construction.  From the proofs of Theorems \ref{thm:main_general} and \ref{thm:cor1}, one basis of the kernel of $Y'A_k'$, which are the associated linear operators for the system $\{\S',\C',\Re'\}$,   has elements, $\{\textbf{b}_1,\dots, \textbf{b}_{t'}\}$, where $\textbf{b}_\theta$ only has support on the $\theta$th terminal strong linkage class of $\{\S',\C',\Re'\}$. By Lemma \ref{thm:kernel} and Proposition 5.1 in \cite{Gunawardena}, we can conclude that the network has a deficiency of zero.
\end{proof}

\subsection{Comments on Theorems \ref{thm:main_general} and \ref{thm:cor1}}
\label{sec:notes}

In this section, we  make a few notes and comments on Theorems \ref{thm:main_general} and \ref{thm:cor1}.  In Section \ref{sec:higher_deficiencies} we demonstrate through an example that there is a difference in content between Theorems \ref{thm:main_general} and \ref{thm:cor1}.  In Section \ref{sec:why_conservation}, we demonstrate through an example the importance of the conservation relation condition in our results.  Finally, in Section \ref{sec:further} we provide  motivation for a conjecture for even more general results pertaining to all ACR models.

\subsubsection{Importance of the kernel condition: higher deficiency models}
\label{sec:higher_deficiencies}

It is clear that both the networks (\ref{system1}) and (\ref{system2}) satisfy the assumptions of Theorem \ref{thm:cor1}.  However, Theorem \ref{thm:cor1}  stands silent on models with a deficiency greater than one.  Here we provide an example of a deficiency two network taken from the  literature that satisfies the assumptions of Theorem \ref{thm:main_general}, demonstrating that there is a true difference in content between the two theorems.\\

\noindent {\it Example 3:} Consider the following generalization of Example 2, in which the phosphatase stimulation effects of both ATP and ADP are simultaneously considered:
	 \begin{equation}
\label{model:def2}
\begin{split}
& XD \mathop{\stackrel{k_1}{\rightleftarrows}}_{k_2[D]} X \mathop{\stackrel{k_3[T]}{\rightleftarrows}}_{k_4} XT \stackrel{k_5}{\rightarrow} X_p \\
& X_p + Y \mathop{\stackrel{k_6}{\rightleftarrows}}_{k_7} X_pY \stackrel{k_8}{\rightarrow} X + Y_p \\
& XD + Y_p \mathop{\stackrel{k_9}{\rightleftarrows}}_{k_{10}} XDY_p \stackrel{k_{11}}{\rightarrow} XD + Y\\
&XT + Y_p \mathop{\stackrel{k_{12}}{\rightleftarrows}}_{k_{13}} XTY_p \stackrel{k_{14}}{\rightarrow} XT + Y_p .
\end{split}
\end{equation}
 This model was shown in \cite{Sh-F} to have a deficiency of two, but still exhibit absolute concentration robustness in $Y_p$.  They showed this through an explicit calculation of the equilibria as opposed to using theoretical results.  As the deficiency of this network is two, Theorem \ref{thm:cor1} is inapplicable.  To apply Theorem \ref{thm:main_general} we note that there are two pairs of complexes satisfying $y \ll y'$:
\[
	XD + Y_p \ll XD, \quad \text{ and } \quad XT + Y_p \ll XT.
\]
That is to say, we have $\mathcal{C}^* = \left\{ XD + Y_p, XT + Y_p \right\}$ and $\mathcal{C}^{**} = \left\{ XD, XT \right\}$ in assumption 1 of  Theorem \ref{thm:main_general}. We consider the following ordering of the non-terminal complexes, which are the only source complexes in this model:
\begin{equation}
\label{333333}
	\{XD,\ X,\ XT,\ X_p + Y,\ X_pY,\ XD+Y_p,\ XDY_p,\ XT+Y_p,\ XTY_p\}.
\end{equation}
Choosing the rate constants $k_1=\cdots=k_{13}=1$,  there are {\it two} basis vectors of $YA_k$ with non-zero support on the complexes as ordered in (\ref{333333}); the components of these basis vectors restricted to the first 9 components (corresponding with the non-terminal complexes) are  
	\[
		\left\{ [2,2,1,2,1,2,1,0,0],[2,2,1,2,1,0,0,2,1] \right\}.
	\]	 
	Since $XD+Y_p$ only  has support on the first basis vector above, and $XT + Y_p$ only has support on the second, we see that assumption 2 of Theorem \ref{thm:main_general} is satisfied, and the conclusion holds.

\subsubsection{Importance of the conservation relation}
\label{sec:why_conservation}

Here we provide an example in which the conclusions of Theorem \ref{thm:cor1} do not hold if only the conservation assumption is dropped.\\

\noindent {\it Example 4:} Consider the chemical reaction network
\begin{equation}
\label{model:non_conserv}
\begin{split}
& A + B \stackrel{\alpha}{\rightarrow} 0 \\
& B \stackrel{\beta}{\rightarrow} A + 2B.
\end{split}
\end{equation}
The network has a deficiency of one,  the non-terminal complexes $A+B$ and $B$ differ only in the species $A$,  and the corresponding deterministic mass-action systems have  positive equilibria so long as $c_A(0)-c_B(0) < \beta/\alpha$. It follows by Theorem \ref{thm:marty} that the corresponding mass-action systems exhibit absolute concentration robustness in $A$. It can be checked explicitly that the equilibrium value is
\[\bar{c}_{A} = \frac{\beta}{\alpha}.\]

The corresponding stochastic chemical reaction system, however, does not necessarily converge to the boundary as predicted by Theorem \ref{thm:main_general}. For any compatibility class $M:= X_{B}-X_{A}$ where $M \geq 1$ we can use results related to birth and death processes to conclude that the stationary distribution for the process tracking the number of $A$ molecules has the form 
\[
	\pi^M(i) = \frac{K^M}{(M+i)i!} \left( \frac{\beta}{\alpha} \right)^i \; \; \; \; \; \; \; \; \; i = 0, 1, 2, \ldots
\]
where $K^M$ is the normalizing constant
\[
	K^M =  \left( \sum_{j=0}^\infty \frac{1}{(M+j)j!} \left( \frac{\beta}{\alpha} \right)^j \right)^{-1}.
\]
Interestingly, it can be seen by standard analytical methods \cite{LittleRudin} that for any $n \geq 0$ we have
\[
\begin{split}
\lim_{M \to \infty} \pi^M(i) & = \left[ \frac{1}{i!} \left( \frac{\beta}{\alpha} \right)^i \right] / \left[ \sum_{j=0}^\infty \frac{1}{j!} \left( \frac{\beta}{\alpha} \right)^j \right]\\
& = \frac{e^{-\frac{\beta}{\alpha}}}{i!} \left( \frac{\beta}{\alpha} \right)^i.
\end{split}
\]
In the limit as $M \to \infty$, therefore, we have that the stationary distribution approaches a Poisson distribution with mean $\beta / \alpha$.

It is easy to see how the argument presented in  Section \ref{proofsection} breaks down. Since the network (\ref{model:non_conserv}) is not conservative, we are not guaranteed that the $\{(A+B)\}$-reduced subnetwork obtained by removing reactions from $A+B$ has a positively recurrent irreducible component. The relevant subnetwork for this example is $B \to A+2B$ which clearly grows unboundedly and therefore does not possess such a component.

\subsubsection{Conclusions hold for further ACR models}
\label{sec:further}

We have already seen with (\ref{model:def2}) that, while the conditions of Theorem \ref{thm:marty} are sufficient for ACR in deterministically-modeled mass-action systems, they are not necessary. It is also the case that, while Theorems \ref{thm:main_general} and \ref{thm:cor1}  capture some ACR systems not captured by Theorem \ref{thm:marty}, they do not capture all such systems.
Nevertheless, a sample of such systems suggests that the conclusions of Theorems \ref{thm:main_general} and \ref{thm:cor1}  hold for a much wider class of conservative ACR networks. We now present one such system.\\

\noindent {\it Example 5:} Consider the following network, which is taken from \cite{N-G-N}:
\begin{equation}
\label{system5}
\begin{split}
A + X & \mathop{\stackrel{k_1}{\rightleftarrows}}_{k_2} F + Y \\
A & \stackrel{k_3}{\to} B \\
C + F \stackrel{k_4}{\to} E & \stackrel{k_5}{\to} D + F \\
B + D & \stackrel{k_6}{\to} A + C \\
X & \mathop{\stackrel{k_7}{\rightleftarrows}}_{k_8} Y.
\end{split}
\end{equation}
The system is governed by the system of differential equations
\begin{equation}
\label{system5de}
\begin{split}
\dot{c}_A & = -k_1c_Ac_X + k_2c_Fc_Y - k_3 c_A + k_6c_Bc_D \\
\dot{c}_B & = k_3c_A - k_6 c_B c_D \\
\dot{c}_C & = -k_4c_Cc_F + k_6c_Bc_D \\
\dot{c}_D & = k_5c_E-k_6c_Bc_D \\
\dot{c}_E & = k_4 c_Cc_F - k_5 c_E \\
\dot{c}_F & = k_1 c_A c_X - k_2 c_F c_Y - k_4 c_C c_F + k_5 c_E \\
\dot{c}_X & = -k_1 c_A c_X + k_2 c_F c_Y - k_7 c_X + k_8 c_Y \\
\dot{c}_Y & = k_1 c_A c_X - k_2 c_F c_Y + k_7 c_X - k_8 c_Y.
\end{split}
\end{equation}
It was shown in \cite{N-G-N} that every positive equilibrium concentration satisfies
\[\bar{c}_C = \frac{k_2 k_3 k_7}{k_1 k_4 k_8}.\]
That is to say, the system (\ref{system5de}) is ACR in the concentration of $C$.

Despite the system being demonstrably ACR in $C$, however, it is not amenable to Theorem \ref{thm:marty}. The network (\ref{system5}) is deficiency one ($\delta = n - \ell - s = 11 - 5 - 5 = 1$) but there do not exist two non-terminal complexes which differ only in the species $C$ (or any species). Since no non-terminal complex is dominated by any other non-terminal complex, we are also unable to apply Theorem \ref{thm:main_general} or Theorem \ref{thm:cor1}.

Nevertheless, we can see that there does not exist a positively recurrent state with mass on any non-terminal complex. To see this, we start by considering the following conservation laws:
\begin{equation}
\label{conservation5}
\begin{split}
T_{XY} & := X(0) + Y(0) \\
T_{ABDF} & := A(0) + B(0) + D(0) + F(0) \\
T_{CDE} & := C(0) + D(0) + E(0).
\end{split}
\end{equation}
Now consider the following chain of reactions from an arbitrary positive starting count in each species:
\begin{enumerate}
\item
Convert all $E$ into $D$ and $F$ through $E \to D+F$.
\item
Convert all $F$ into $A$ through $F+Y \to A+X$ (replenishing $Y$ with $X \to Y$ as necessary).
\item
Convert all $D$ into $C$ through $B + D \to A +C$ (replenishing $B$ with $A \to B$ as necessary).
\item
Convert all $A$ into $B$ through $A \to B$.
\end{enumerate}
From the conservation laws (\ref{conservation5}), the final state of this algorithm is $\bar{X}_A = 0$, $\bar{X}_B = T_{ABDF} > 0$, $\bar{X}_C = T_{CDE} > 0$, $\bar{X}_D=0$, $\bar{X}_E=0$, $\bar{X}_F=0$, $X_X>0$, $X_Y>0$, from which no reaction from a non-terminal complex may proceed. In fact, the only reactions which may proceed are $X \to Y$ and $Y \to X$. It follows that, with a probability of one, there is a final time for which any chain has positive mass on any non-terminal complex. This is far removed from the deterministic prediction for the long-term value of $X_C$.

We therefore make the following conjecture:

\begin{conjecture}
\label{conjecture}
Consider a conservative chemical reaction network $(\mathcal{S},\mathcal{C},\mathcal{R})$ and associated stochastic chemical reaction system with some stoichiometrically admissible kinetics. Suppose that for some choice of rate constants $\{k_i\}_{i = 1}^r$ the deterministically modeled mass-action system exhibits ACR in some species $X_i$. Then the conclusions of Theorems \ref{thm:main_general} and  \ref{thm:cor1} hold.
\end{conjecture}

\section{Time until absorption and quasi-stationary distributions}
\label{analysissection}

 If the time to absorption is large relative to the relevant time-scales of the system, the processes will seem to settle down to an equilibrium long before the resulting instability will appear.  This limiting distribution is called a \textit{quasi-stationary} distribution, and we refer the reader to \cite{D-S} for a proof of that fact that such distributions exists, and are unique, in the current setting.  The resulting distribution  bridges the gap between the extinction event and the transient behavior of the process.

In this section, we briefly introduce the notation and background results relevant to the study of quasi-stationary distributions, though we refer the interested reader to \cite{Meleard_Quasi} for a recent survey on quasi-stationary distributions and population processes for a thorough introduction to the topic, including bibliography. We  apply the results to the activation/deactivation network (\ref{system1}) and the  EnvZ/OmpR signaling network (\ref{system2}).

\subsection{Quasi-stationary probability distributions}
\label{stabilitysection}

We  divide the state space $\mathbb{Z}_{\geq 0}^m$ into the set $\mathbf{X}_T$ of transient states and $\mathbf{X}_A$ of absorbing states.  There are two related notions that capture the relevant long-term transient dynamics which we take from \cite{Meleard_Quasi}.  

\begin{definition}
 We say that the process $\mathbf{X}(t)$ has a \textit{Yaglom limit} if there exists a probability distribution $\tilde \pi$ on $\mathbf{X}_T$ such that, for any $\mathbf{X,Y}\in \mathbf{X}_T$,
	\[
		\lim_{t\to \infty} P(\X(t) = \mathbf{Y} \ | \ \X(t) \notin \X_A, \X(0) = \X) = \tilde \pi(\mathbf{Y}).
	\]
\end{definition}
\begin{definition}
\label{def:quasi-stationary}
Let $\tilde \pi$ be a probability distribution on $\X_T$.  We say that $\tilde \pi$ is a {\it quasi-stationary  distribution} (QSD) if, for all $t \ge 0$ and $\mathbf{Y} \in \X_T$
\begin{equation}
\label{quasi1}
	\tilde \pi(\mathbf{Y}) = P_{\tilde \pi}(\X(t) = \mathbf{Y}| \X(t) \notin \X_A),
\end{equation}
where $P_{\tilde \pi}$ is the distribution of the process given an initial distribution of  $\tilde \pi$.   
\end{definition}

\noindent In other words, the Yaglom limit is the limiting distribution of a chain conditioned on not entering the absorbing set, while the quasi-stationary probability distribution is the distribution which remains unchanged under the condition that the associated chain does not enter a state in the absorbing set. In the current setting of a continuous time Markov chain with a finite state space, the existence, uniqueness, and equivalence of the Yaglom limit and QSD is established \cite{D-S,Meleard_Quasi}.

A sometimes useful formula for the quasi-stationary distribution $\tilde{\pi}$ can be derived by first introducing the conditional probabilities
\begin{equation}
\label{8392}
Q_t(\X) := P( \mathbf{X}(t) = \mathbf{X} \; | \; \mathbf{X}(t) \not\in \mathbf{X}_A ) = \frac{P_t(\X)}{1 - P_t(\X_A)}
\end{equation}
where $P_t(\X_A): = P(\X(t) \in \X_A)$ is the probability the chain has entered an absorbing state by time $t$. In the present setting of a CTMC with finite state space we may differentiate \eqref{8392} and collect terms appropriately to see that the evolution of the conditioned probabilities $Q_t$ satisfy a system of ordinary differential equations.   Setting the resulting set of equations to zero, substituting $\tilde \pi$ for $Q_t$, and solving yields the system
\begin{equation}
\label{98e}
\tilde{\pi} \A_Q = \theta \tilde{\pi},
\end{equation}
where $\A_Q$ is the restriction of the jump rate matrix $\A$ given in (\ref{def:cme}) to the transient states $\mathbf{X}_T$.  The term $\theta$ is negative and can be explicitly computed \cite{Meleard_Quasi}, but is not important for our purposes.  A simple calculation using the Perron-Frobenius Theorem shows that there is a unique solution to \eqref{98e} with strictly positive components, which is also the unique quasi-stationary distribution; see, for example,  \cite{Collet2013,Nasell1996}.

Consideration of the quasi-stationary distribution $\tilde{\pi}$ is most appropriate when the chain $\mathbf{X}(t)$ is expected to spend significant time in the transient region before converging to the absorbing set. In order to quantify this hitting time, it is typical to define
\[\tau_A(\mathbf{X}) := \inf ( t \geq 0 \; | \; \mathbf{X}(t) \in \mathbf{X}_A, \mathbf{X}(0) = \mathbf{X} )\]
to be the first time the chain $\mathbf{X}(t)$ with initial condition $\mathbf{X}(0) = \mathbf{X}$ enters the absorbing set. The quantity $\mathbbm{E}(\tau_A(\mathbf{X}))$ is called the {\it expected time to absorption}. Explicit formulas for $\mathbbm{E}(\tau_A(\mathbf{X}))$ for continuous time Markov chains with countable state spaces may be found in Section 6.7 of \cite{Allen}.

\subsection{Activation/deactivation system}
\label{convergencesection}

Reconsider the activation/deactivation network (\ref{system1}). We have demonstrated previously that the deterministically modeled mass-action system (\ref{de}) exhibits ACR in the species $A$ while trajectories of the stochastic chemical reaction system converge almost surely to the boundary state (\ref{system1boundary}) in finite time. In order to search for ACR like behavior in the stochastic setting, therefore, we consider the quasi-stationary distribution when the expected time until absorption, $\mathbbm{E}(\tau_A(\mathbf{X}))$, is high.

\begin{figure}[h]
\vspace{-0.5cm}
\begin{center}
\includegraphics[width=15cm]{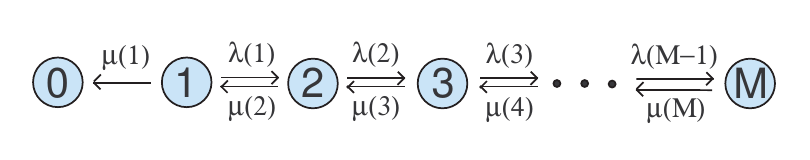}
\end{center}
\vspace{-1cm}
\caption{Birth-death chain corresponding to the activation/deactivation network (\ref{system1}). The number $i$ corresponds to the count of $X_B$ and the values $\lambda(i)$ and $\mu(i)$ correspond to the transition propensities of the first and second reaction, respectively, from the state $X_B=i$. The chain has a unique sink at $i=0$.}
\label{supp-figure1}
\end{figure}

We consider the stochastic model corresponding to (\ref{system1}) for a fixed value $M:= X_A(0) + X_B(0) > 0$. Notice that determining the behavior of the chain $X_B(t)$ is sufficient to determine the chain on $A$ by the conservation relationship $X_A(t) = M - X_B(t)$. Also notice that the first reaction in (\ref{system1}) increases $X_B$ by one while the second reaction decreases $X_B$ by one. Consequently, the network (\ref{system1}) corresponds to a birth-death chain with an absorbing state at $X_B=0$. For a graphical illustration of the state space of the chain $X_B$, see Fig. \ref{supp-figure1}. The propensities $\lambda(i)$ and $\mu(i)$ are given by 
\begin{equation}
\label{propensities}
\begin{split}
\lambda(i) & = \alpha X_B X_A = \alpha i ( M - i), \\
\mu(i) & = \beta X_B = \beta i.
\end{split}
\end{equation}

The chain $X_B(t)$ with transition propensities (\ref{propensities}) has been considered in a number of different contexts. The chemical reaction network (\ref{system1}) is considered explicitly in \cite{O-S-W} where only the labeling of the species differs. A more extensive analysis is given in the papers \cite{A-B,K-L,Nasell1996,W-D} where the chain $X_B(t)$ corresponds to chain of infected individuals $I(t)$ in the susceptible-infected-susceptible (SIS) epidemic model. The model also arises in stochastic logistic population growth models on a finite state space \cite{Nasell1999,Nasell2001,Norden1982}.

The expected time until absorption has an explicit formula for birth-death chains with a countable state space \cite{Allen}. 
We let $\tau_i$ be shorthand for the time to absorption from the state $X_B = i$.  Substituting the propensities (\ref{propensities}) into equation (6.22) of \cite{Allen} and simplifying yields
\begin{equation}
\label{solution2}
\mathbbm{E}(\tau_i) = \displaystyle{\sum_{k=0}^{i-1} \sum_{j=k+1}^M \frac{1}{\beta} \left( \frac{\alpha}{\beta} \right)^{j-k-1} \frac{(M-k-1)!}{j (M-j)!}}
\end{equation}
for $i=1, \ldots, M$. The numerical results contained in Table \ref{table1} justify the consideration of the quasi-stationary distribution $\tilde{\pi}$ for values of $M$ where $\mathbbm{E}(\tau_i)$ is large.

\begin{table}
    \centering
	\begin{subtable}[b]{0.3\textwidth}
	\centering
	\begin{tabular}{c|c}
		$M$ & $\mathbbm{E}(\tau_1)$ (s) \\
		\hline
		\hline
		$5$ & $0.0438$ \\
		$10$ & $0.0491$ \\
		$15$ & $0.0572$ \\
		$20$ & $0.0699$ \\
		$25$ & $0.0930$ \\
		\hline
	\end{tabular} \vspace{0.15in}
	\caption{Small-scale}
	\end{subtable}
	\begin{subtable}[b]{0.3\textwidth}
	\centering
	\begin{tabular}{c|c}
		$M$ & $\mathbbm{E}(\tau_1)$ (s) \\
		\hline
		\hline
		$30$ & $0.147$ \\
		$35$ & $0.332$ \\
		$40$ & $1.412$ \\
		$45$ & $12.913$ \\
		$50$ & $233.051$ \\
		\hline
	\end{tabular} \vspace{0.15in}
	\caption{Mid-scale}
	\end{subtable}
	\begin{subtable}[b]{0.3\textwidth}
	\centering
	\begin{tabular}{c|c}
		$M$ & $\mathbbm{E}(\tau_1)$ (s) \\
		\hline
		\hline
		$55$ & $7.42 \times 10^{3}$ \\
		$60$ & $3.88 \times 10^{5}$ \\
		$65$ & $3.16 \times 10^{7}$ \\
		$70$ & $3.87 \times 10^{9}$ \\
		$75$ & $6.87 \times 10^{11}$ \\
		\hline
	\end{tabular} \vspace{0.15in}
	\caption{Large-scale}
	\end{subtable}
	\caption{Expected time until absorption for the activation/deactivation system (\ref{system1}) from the state $(X_A,X_B)=(M-1,1)$ with parameter values $\alpha = 1$, $\beta = 25$. Note that convergence to a positive equilibrium concentration is predicted in the deterministic model for $M > 25$ but that the expecting hitting time for the absorbing state (\ref{system1boundary}) grows only modestly in the range Mid-scale range of $M$.}
	\label{table1}
	\hspace*{\fill}
\end{table}

A numerical iteration scheme for approximating the the quasi-stationary distribution $\tilde{\pi}$ of a finite-state birth and death process is given in \cite{Cavender1978} and adapted to the stochastic SIS model in \cite{Nasell1996}. Relabeling and rearranging the relevant parameters to fit (\ref{propensities}), we have that the quasi-stationary distribution $\tilde{\pi}$ of the CTMC corresponding to (\ref{system1}) with propensities (\ref{propensities}) can be approximated as the limit of the iterative scheme
\begin{equation}
\label{iterative}
\begin{split}
\tilde{\pi}_j(1) & = \sum_{i=1}^M \left[ \frac{1}{i} \left( \sum_{k=1}^{i} \frac{(M-k)!}{(M-i)!} \left( \frac{\beta}{\alpha} \right)^{i-k} \left( 1 - \sum_{l = 1}^{k-1} \tilde{\pi}_{j-1}(l) \right) \right) \right] \\
\tilde{\pi}_j(i) & = \frac{1}{i} \left[ \sum_{k=1}^{i} \frac{(M-k)!}{(M-i)!} \left( \frac{\beta}{\alpha} \right)^{i-k} \left( 1 - \sum_{l = 1}^{k-1} \tilde{\pi}_{j-1}(l) \right) \right] \tilde{\pi}_{j-1}(1), \; \; \; i=2, \ldots, M.
\end{split}
\end{equation}
That is to say, we have $\tilde{\pi}_j(i) \to \tilde{\pi}(i)$ as $j \to \infty$ \cite{Cavender1978}. To obtain the quasi-stationary distribution in the ACR species $A$, we need to invert $\tilde{\pi}$. Hence, we compute the distribution $\tilde{\pi}_A$ which has entries $\tilde{\pi}_{A}(i) = \tilde{\pi}(M-i)$, $i=0, \ldots, M-1$. The results are shown in Fig. \ref{fig:quasi-stationary}.


\begin{wrapfigure}{r}{4in}
        \centering
        \includegraphics[width=0.6\textwidth]{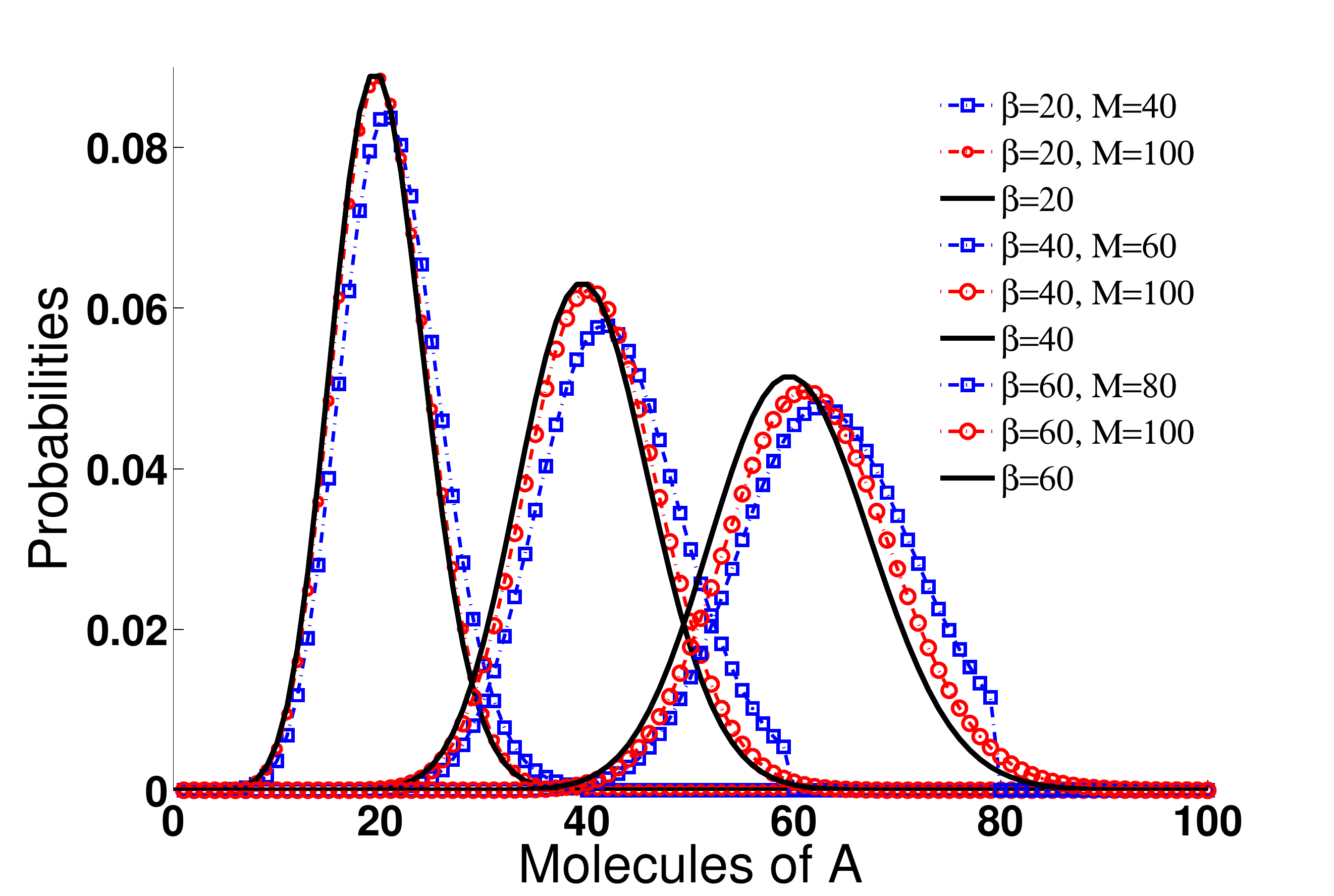}
        \caption{Quasi-stationary distributions of $X_A$ with $\alpha = 1$ and various values of $\beta$ and $M$. As $M \to \infty$ the quasi-stationary distributions approaches the overlain Poisson distributions \eqref{poisson} (solid line).}
	\label{fig:quasi-stationary}
\end{wrapfigure}

A striking feature of the quasi-stationary distributions shown in Fig. \ref{fig:quasi-stationary} is that they appear to approach the Poisson distribution 
\begin{equation}
\tilde{\pi}(i) = e^{-\left( \frac{\beta}{\alpha} \right)} \frac{\left( \frac{\beta}{\alpha} \right)^i}{i!},
\label{poisson}
\end{equation}
 as $M \to \infty$. It has recently been shown that the ``one permanently infected'' approximation of the quasi-stationary distribution approaches the Poisson distribution (\ref{poisson}) in this limit \cite{H-C-L-H}. For completeness, we now prove that the true quasi-stationary distribution approaches a Poisson distribution.  

\begin{lemma}
\label{poissonproof}
Consider the stochastically modeled activation/deactivation network (\ref{system1}) with mass-action propensities (\ref{propensities}). The quasi-stationary distribution of $X_A$ converges to the Poisson distribution (\ref{poisson}) in the limit $M \to \infty$.
\end{lemma}
\begin{proof}
We first of all reformulate the propensities (\ref{propensities}) to correspond to the chain tracking the number of $A$ molecules. Paremeterizing the model by $M$, and keeping the convention that $\lambda^M$ refer to ``birth'' and $\mu^M$ refers to ``death'' (now with respect to species $A$), we have
\begin{equation}
\label{propensities2}
\begin{split} \lambda^M(i) & = \beta (M-i) \\ \mu^M(i) & = \alpha i(M-i).\end{split}
\end{equation}
Note that this chain has a reflecting state at $X_A = 0$ and an absorbing state at $X_A = M$. The transient states are $\left\{ 0, \ldots, M-1 \right\}$ so that the generator of (\ref{def:cme}) of the Markov process conditioned on non-absorption (the $\A_Q$ in (\ref{98e})) has support on these states. By (\ref{98e}) we have that the quasi-stationary distribution for a given $M$ satisfies 
\begin{align}
\begin{split}
	-\lambda^M(0) \tilde \pi_0^M + \mu^M(1)\tilde \pi_{1}^M &= \theta^M \tilde \pi_0^M\\
	\lambda^M(i-1) \tilde{\pi}^M_{i-1} - (\lambda^M(i) + \mu^M(i)) \tilde{\pi}^M_i + \mu^M(i+1) \tilde{\pi}^M_{i+1} &= \theta^M \tilde{\pi}^M_i\\
 -\lambda^M(M-2) \tilde \pi_{M-2}^M + \mu^M(M-1)\tilde \pi_{M-1}^M &= \theta^M \tilde \pi_{M-1}^M
\end{split}
\label{989898}
\end{align}
where $i=1, \ldots, M-2$ and $\theta^M$ is the eigenvalue corresponding to the eigenvector equation (\ref{98e}) for the given $M$. It can be directly computed that $\theta^M = \beta \tilde{\pi}_M^M$, which we note is uniformly bounded in $M$ \cite{Nasell1996}. Substituting $\theta^M = \beta \tilde \pi_M^M$ and (\ref{propensities2}) into (\ref{989898}) yields
\begin{align*}
\begin{split}
	-\beta M \tilde \pi_0^M + \alpha(M-1)\tilde \pi_{1}^M &= \beta \tilde{\pi}_M^M \tilde \pi_0^M\\
	\beta(M-(i-1))\tilde{\pi}^M_{i-1} - (\beta(M-i) + \alpha i (M-i)) \tilde{\pi}^M_i + \alpha (i+1)(M-(i+1))\tilde{\pi}^M_{i+1} &= \beta \tilde{\pi}_M^M \tilde{\pi}^M_i,\\
	 -2 \beta \tilde \pi_{M-2}^M + \alpha \tilde \pi_{M-1}^M & = \beta \tilde{\pi}_M^M \tilde \pi_{M-1}^M.
\end{split}
\end{align*}
Dividing by $M$ yields
\begin{align*}
-\beta  \tilde \pi_0^M + \alpha \left(1-\frac{1}{M}\right)\tilde \pi_{1}^M &= \frac{ \beta \tilde{\pi}_M^M \tilde \pi_0^M}{M}\\
	\beta\left(1-\frac{(i-1)}{M}\right)\tilde{\pi}^M_{i-1} - \left(\beta \left(1-\frac{i}{M}\right) + \alpha i \left(1-\frac{i}{M}\right)\right) \tilde{\pi}^M_i &\\
	+ \alpha (i+1)\left(1-\frac{(i+1)}{M}\right)\tilde{\pi}^M_{i+1} &= \frac{\beta \tilde{\pi}_M^M \tilde{\pi}^M_i}{M},
	\end{align*}
for $i=1, \ldots, M-2$.   The use of standard limiting arguments shows that as $M\to \infty$, the vector $\tilde \pi^M$ converges to the solution of the difference equations
\begin{align*}
	-\beta   \pi_0 + \alpha  \pi_{1} &=0\\
	\beta \pi_{i-1} - \left(\beta + \alpha i \right) \pi_i + \alpha (i+1) \pi_{i+1} &= 0,
\end{align*}
where $i\geq 1$, subject to the constraint $\sum_{i=0}^\infty \pi_i = 1$. This is known to be Poisson with parameter $\beta/\alpha$ yielding (\ref{poisson}) \cite{Lawler}.
\end{proof}

\subsection{EnvZ/OmpR signaling system}

Reconsider the proposed EnvZ/OmpR signal transduction system (\ref{system2}). Since the network satisfies the assumptions of Theorem \ref{thm:marty} and is furthermore conservative by (\ref{922}), it follows by Theorem \ref{thm:cor1} that trajectories $\mathbf{X}(t)$ converge almost surely in finite time to an absorbing set near the boundary. 

\begin{wraptable}{r}{7cm}
	\centering
	\footnotesize{
	\begin{tabular}{|rl|rl|}
		\hline
		$k_1$ & $0.5$ $s^{-1}$ & $\kappa_1$ & $0.5$\\
\rowcolor[gray]{0.9}		$k_2$ & $0.5$ $s^{-1}$ & $\kappa_2$ & $0.5$\\
		$k_3$ & $0.5$ $s^{-1}$ & $\kappa_3$ & $0.5$\\
\rowcolor[gray]{0.9}		$k_4$ & $0.5$ $s^{-1}$ & $\kappa_4$ & $0.5$\\
		$k_5$ & $0.1$ $s^{-1}$ & $\kappa_5$ & $0.1$\\
\rowcolor[gray]{0.9}		$k_6$ & $0.5$ $\mu M^{-1} s^{-1}$ & $\kappa_6$ & $0.02$\\
		$k_7$ & $0.5$ $s^{-1}$ & $\kappa_7$ & $0.5$\\
\rowcolor[gray]{0.9}		$k_8$ & $0.5$ $s^{-1}$ & $\kappa_8$ & $0.5$\\
		$k_9$ & $0.5$ $\mu M^{-1} s^{-1}$ & $\kappa_9$ & $0.02$\\
\rowcolor[gray]{0.9}		$k_{10}$ & $0.5$ $s^{-1}$ & $\kappa_{10}$ & $0.5$\\
		$k_{11}$ & $0.1$ $s^{-1}$ & $\kappa_{11}$ & $0.1$\\
		\hline
\rowcolor[gray]{0.9}		$[D]$ & $1$ $\mu M$ & $n_A$ & $6.022 \times 10^{23}$ \\
		$[T]$ & $1$ $\mu M$ & $V$ & $4.151 \times 10^{-17}$ $L$ \\
		\hline
	\end{tabular} \vspace{0.15in}
	}
	\caption{Parameter values used for numerical simulations of (\ref{system2}).}
	\label{parameters}
\end{wraptable} 

To further analyze the system, we would like to compute the expected times until absorption $\mathbbm{E}(\tau_A(\mathbf{X}))$ and the quasi-stationary distribution $\tilde{\pi}$. The complexity of the network, however, prohibits explicit derivation of either of these quantities. We choose instead to approximate these quantities by numerical simulation. The parameter values we use for all simulations are given in Table 2 and are in close agreement with those contains in Table 1 of \cite{Igoshin2008}. We note several differences. For instance, the authors of \cite{Igoshin2008} do not consider the mechanism for ADP or ATP binding with the source kinase, EnvZ. Neither do they consider the EnvZ-ADP compound as a regulator for the phosphorylation of OmpR. We have filled in these rate constants  with values of the same order as the known rates of the system.


In order to convert the deterministic rate parameters into stochastic rate parameters, we adjust the second-order reactions ($k_6$ and $k_9$) by a factor of $\kappa_i = k_i / (n_A V)$ where $n_A$ is Avogadro's number and $V$ is the volume of the {\it E. coli} cell. The volume of an {\it E. coli} cell has been estimated at $V =10^{-15}$ $L$ \cite{Cai2002}; for numerical simplicity, however, we choose $V = 25 / n_A$ so that $n_A V = 25$. After converting micromolar units $\mu M$ to molar units $M$, this gives a value of $V = 4.151 \times 10^{-17}$ $L$. It can be easily checked in (\ref{cyp}) that the deterministic parameter values in Table \ref{parameters} give
\[\bar{c}_{Y_p} = \frac{k_1k_3k_5(k_{10}+k_{11})[T]}{k_2(k_4+k_5)k_9k_{11}[D]} = 1 \; \mu M.\]
In a cell of volume $V$ given above, this corresponds to a molecular count of
\[\bar{Y}_p = n_A V \bar{c}_{Y_p} = 25\]
for the corresponding stochastic model. We therefore expect the marginal distribution of the quasi-stationary distribution in $Y_p$ to have a mean of roughly $25$ molecules. We furthermore note that there are two conservation relationships to consider, one in the signaling protein EnvZ ($X$), and one in the response regulator OmpR ($Y$). It was found in \cite{Cai2002} that a typical {\it E. coli} cell has roughly $100$ total molecules of EnvZ and $3500$ total molecules of OmpR. Therefore, for our estimations of $\tilde{\pi}$, we chose a ratio of $Y_{tot}:X_{tot}$ of $35:1$, where $X_{tot}$ and $Y_{tot}$ are defined in \eqref{922}.

 We use several numerical estimators for the quantities of interest.  Throughout, we denote by $\left\{ \mathbf{X}_{[i]}(t) \right\}_{i = 1}^N$ an ensemble of trajectories, where $\X_{[i]}$ is the $i$th independent trajectory.  All simulations are carried out using Gillespie's algorithm \cite{Gillespie}.   We note that the state $\X_A$ is the absorption state given in \eqref{system2boundary}.

\begin{figure*}
   \centering
	\footnotesize{
	\begin{subtable}[b]{0.2\textwidth}
	\centering
	\begin{tabular}{c|c}
		$Y_{tot}$ & $\mathbbm{E}(\tau_A(\mathbf{X}_0))$\\ 
		& (in $s$) \\
		\hline
		\hline
		$10$ & $9.3785$ \\
		$12$ & $31.215$ \\
     	$14$ & $56.022$ \\
     	$16$ & $86.413$ \\
     	$18$ & $123.26$ \\
     	$20$ & $167.19$ \\
     	$22$ & $221.67$ \\
     	$24$ & $294.05$ \\
		\hline
	\end{tabular} \vspace{0.15in}
	\caption{Small-scale}
	\end{subtable}
	\qquad
	\begin{subtable}[b]{0.2\textwidth}
	\centering
	\begin{tabular}{c|c}
		$Y_{tot}$ & $\mathbbm{E}(\tau_A(\mathbf{X}_0))$ \\
		& (in $s$) \\
		\hline
		\hline
     	$26$ & $382.73$ \\
     	$28$ & $500.94$ \\
     	$30$ & $668.34$ \\
     	$32$ & $905.29$ \\
     	$34$ & $1270.7$ \\
     	$36$ & $1912.7$ \\
     	$38$ & $3158.6$ \\
     	$40$ & $5702.1$ \\
		\hline
	\end{tabular} \vspace{0.15in}
	\caption{Mid-scale}
	\end{subtable}
	}
	\qquad
	\begin{subfigure}[b]{.35\textwidth}
	\centering
	\includegraphics[width=\textwidth]{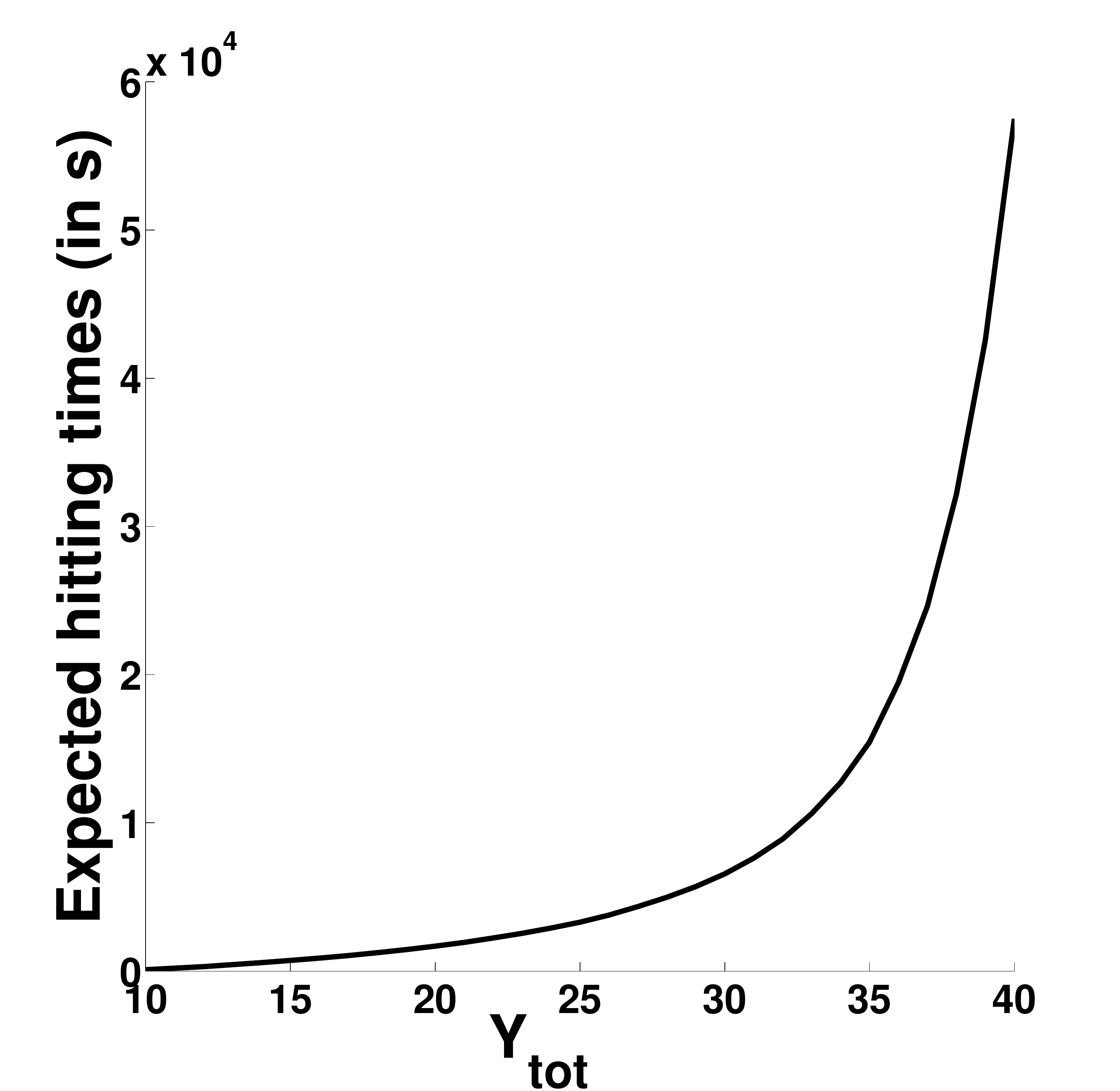}
	\caption{Plot of expectations}
	\end{subfigure}

	\caption{Expected time until absorption for various values of $Y_{tot}$ when $X_{tot}=1$.}
	\label{figure19}
	\hspace*{\fill}
\end{figure*}
 
 We outline the estimators here:
\begin{enumerate}
	\item
	The estimator for the expected time until absorption is
\begin{equation}
\label{estimator1}
\mathbbm{E}(\tau_A(\mathbf{X})) \approx \frac{1}{N} \sum_{i=1}^N T_i
\end{equation}
where $T_i = \inf \left\{ t \geq 0 \; | \; \mathbf{X}_{[i]}(T) = \mathbf{X}_A, \mathbf{X}_{[i]}(0) = \mathbf{X} \right\}$ and $\mathbf{X}$ is the state
\begin{equation}
\label{oneaway}
\begin{split}
Y_p & = Y_{tot} \\
X_p & = X_{tot}-1 \\
XT & = 1 \\
XD & = X = X_pY = Y = XDY_p = 0.
\end{split}
\end{equation}
Note that a single occurrence of the reaction $XT \to X_p$ takes (\ref{oneaway}) to (\ref{system2boundary}). That is to say, we compute the average time to absorption for an ensemble of realizations which start one step away from the absorbing state.
\item
The estimator for  the probability distribution of $Y_p$ at a fixed time $T$ from initial distribution $\nu$ is
\begin{equation}
\label{estimator3}
P_\nu(Y_p(t) = j) \approx \frac{1}{N} \sum_{i=1}^{N} \mathbf{1}(Y_{p,[i]}(T) = j)
\end{equation}
where $Y_{p,[i]}(T)$ is the value of species $Y_p$ at time $T$ for the $i$th simulation, seeded according to the distribution $\nu$. 
\item
The estimator for the  quasi-stationary distribution of $Y_p$ at time $T$ is
\begin{equation}
\label{estimator2}
\tilde{\pi}(Y_p = j) \approx \frac{1}{N_s} \sum_{i=1}^{N} \mathbf{1}(Y_{p,[i]}(T) = j) \cdot \mathbf{1}(\mathbf{X}_{[i]}(T) \not= \mathbf{X}_A).
\end{equation}
where $\mathbf{X}_A$ is the state (\ref{system2boundary}), $Y_{p,[i]}(T)$ is the value of $Y_p$ at time $T$ for the $i$th simulation, and $N_s\le N$ is the number of the $N$ simulations which did not enter the absorbing state (\ref{system2boundary}) by time $T$. In other words, we compute the ensemble $\left\{ \mathbf{X}_{[i]}(t) \right\}$ up to a fixed time $T$ and record the value only if the chain has not entered the absorbing state (\ref{system2boundary}). We then average over these non-absorbed trajectories.

\item In order to approximate the quasi-stationary distribution for extremely large values of $X_{tot}$ and $Y_{tot}$, which are required for Fig. \ref{figure21}, the ensemble estimator of \eqref{estimator2} proved to be infeasible.  Instead, we approximated the model by setting the rate of the transition to the absorbing state to zero, thereby producing an ergodic process which closely follows the dynamics of the actual process.  We then used time-averaging techniques to estimate the resulting  stationary distribution.  See  \cite{BreyerRoberts1999} for a connection between the time averaged process and the quasi-stationary distribution.  We note that in all our simulations using this method for the values of $X_{tot}$ and $Y_{tot}$ as given in Fig. \ref{figure21}, we never observed this change in the dynamics playing a role in the simulation.  That is, the paths never would have made the transition to the absorbing state even if we had allowed such a transition to occur.
\end{enumerate}

We note that the estimators (\ref{estimator1}) and (\ref{estimator3}) are known to converge almost surely in the limit $N \to \infty$.    The estimator (\ref{estimator2}) produces the Yaglom limit as $N\to \infty$ followed by $T\to \infty$.  We know this limit to be the quasi-stationary distribution in the present setting. 
\begin{figure*}[t]
   \centering
	\begin{subfigure}[b]{.45\textwidth}
	\centering
	\includegraphics[width=\textwidth]{Probabilities-Matt-mod.pdf}
	\caption{Evolution from $X = X_{tot}$, $Y = Y_{tot}$}
	\end{subfigure}
    \qquad
	\begin{subfigure}[b]{.45\textwidth}
	\centering
	\includegraphics[width=\textwidth]{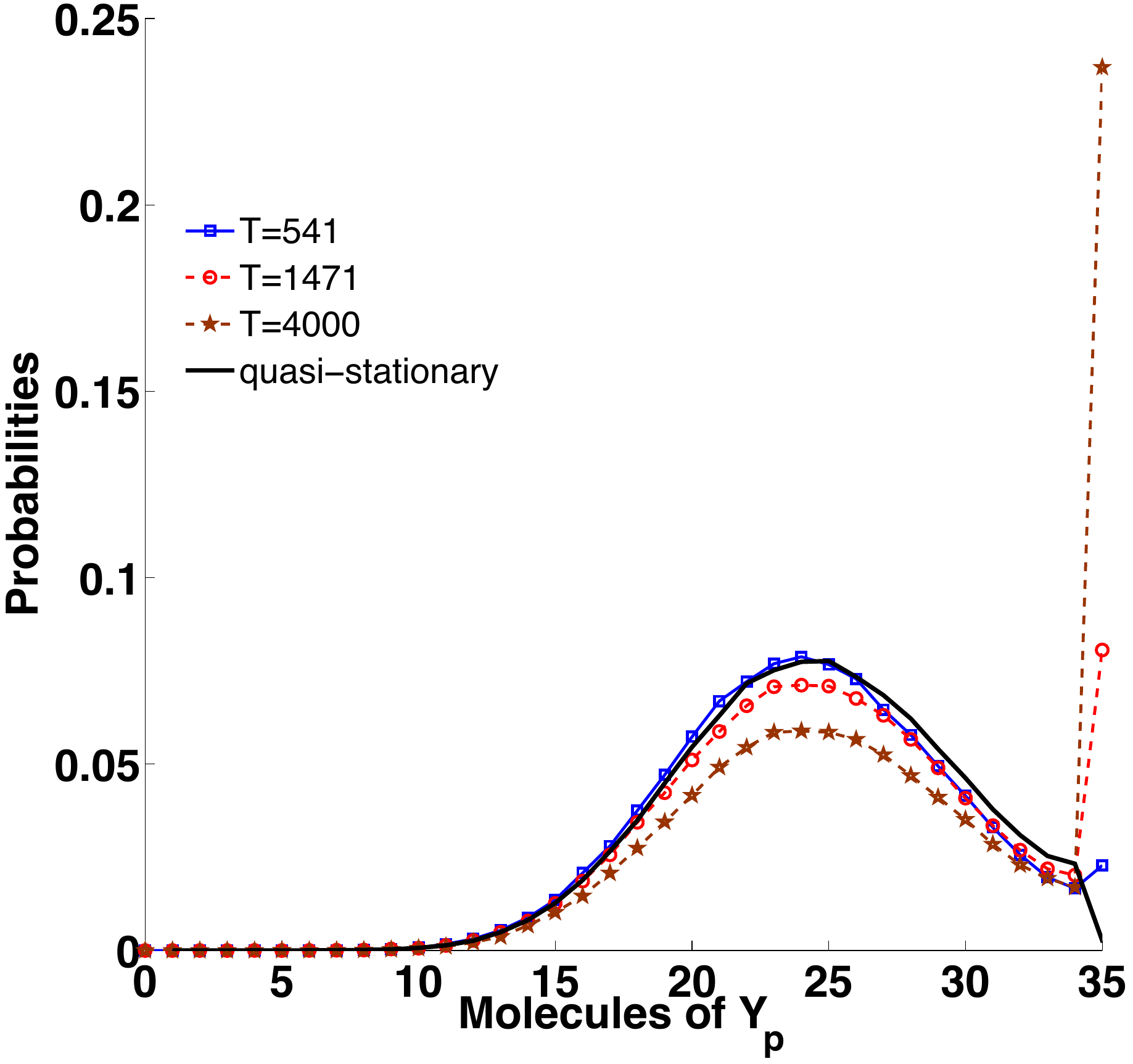}
	\caption{Evolution from quasi-stationarity}
	\end{subfigure}
	\caption{Distributions as a function of time with $X_{tot} =1$ and $Y_{tot} = 35$.  In (a), a point mass is used as the initial distribution, whereas in (b) the initial distribution is the quasi-stationary distribution.}
	\label{figure20}
	\hspace*{\fill}
\end{figure*}

\begin{wrapfigure}{l}{3.3in}
	\centering
	\includegraphics[width=0.47\textwidth]{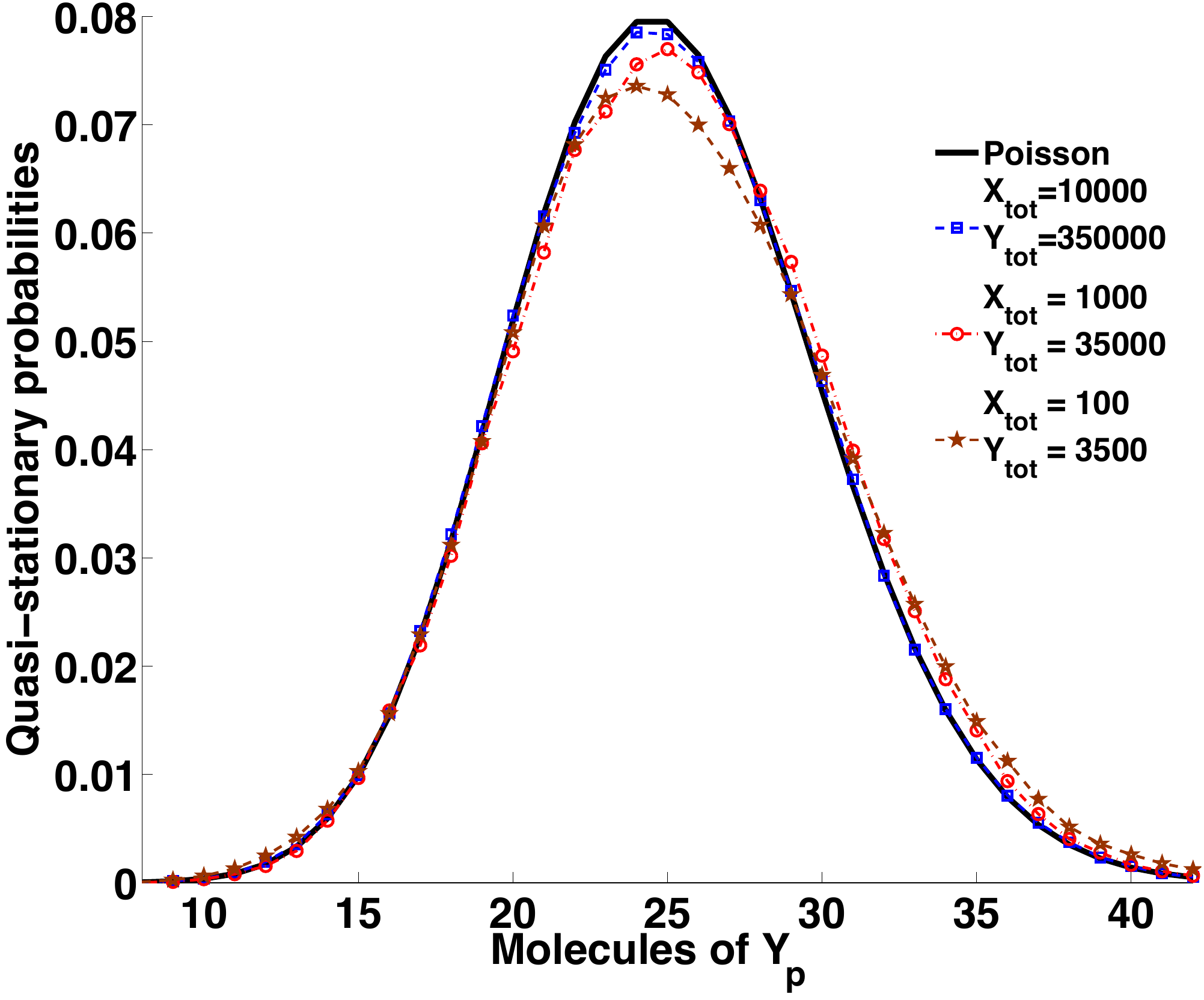}
	\caption{Approximations of the  quasi-stationary distribution of $Y_p$ using the estimator (\ref{estimator2}). The ratio of $Y_{tot} : X_{tot}$ of $35 : 1$ is maintained throughout. For comparison, the Poisson distribution with mean $25$ is overlain (black). }
	\label{figure21}
	\hspace*{\fill}
\end{wrapfigure}

In Fig. \ref{figure19}, we provide approximations of the expected time until absorption $\mathbbm{E}(\tau_A(\mathbf{X}_0))$ for various values of $Y_{tot}$ when $X_{tot}=1$.  Ten thousand data points where sampled for each value of $Y_{tot}$ using the estimator (\ref{estimator1}).  Note that the deterministic model predicts an attracting boundary state for the values in table (a) and an attracting positive equilibrium for the values in table (b). The large expected times until absorption as $Y_{tot}$ grows is clear and justifies the consideration of quasi-stationary distributions.

In Figure \ref{figure20}, we estimate the distribution of the (non-conditioned) process at various times for $X_{tot}=1$ and $Y_{tot}=35$. The estimator \eqref{estimator3} is used to produce the time-dependent distributions of the original process, whereas the estimator \eqref{estimator2} is used to produce the quasi-stationary distribution (using a time $T = 10,000$).  In (a), the initial state is $X(0) = X_{tot}$, $Y(0) = Y_{tot}$ while, in (b), the initial distribution is the quasi-stationary distribution of the process.  Convergence toward the absorbing state (\ref{system2boundary}) is clear. It is also worth noting that the shape of the distribution in (b) away from $Y_p=35$ remains constant, as expected from the definition (\ref{8392}). $100,000$ sample trajectories were computed for the estimates in (a) while $300,000$ sample trajectories were computed for (b).

In Fig. \ref{figure21} we provide approximations of the quasi-stationary distribution of $Y_p$ using the estimator (\ref{estimator2}). The convergence to Poisson (in black) is striking, 
but may only be conjectured at the current time.

%
%

\end{document}